\documentclass[11pt,a4paper]{amsart}
\usepackage[margin=2.8cm]{geometry}
\usepackage{amssymb}
\usepackage{graphicx} 
\usepackage{mathrsfs}
\usepackage{enumerate}
\usepackage{xspace}
\usepackage{color}
\usepackage{enumerate}
\usepackage{enumitem}
\usepackage{amsthm}
\usepackage{dsfont}
\usepackage{bbm}
\usepackage[colorlinks=true, linkcolor = blue, citecolor = blue]{hyperref}
\usepackage[numbers]{natbib}
\usepackage[backgroundcolor=white,bordercolor=red]{todonotes}
\DeclareMathAlphabet{\mathpzc}{OT1}{pzc}{m}{it}
\usepackage[nameinlink]{cleveref}
\numberwithin{equation}{section}
\begin{document}

\theoremstyle{plain}

\newtheorem{theorem}{Theorem}[section]
\newtheorem{lemma}[theorem]{Lemma}
\newtheorem{example}[theorem]{Example}
\newtheorem{proposition}[theorem]{Proposition}
\newtheorem{corollary}[theorem]{Corollary}
\newtheorem{definition}[theorem]{Definition}
\newtheorem{Ass}[theorem]{Assumption}
\newtheorem{condition}[theorem]{Condition}
\theoremstyle{definition}
\newtheorem{remark}[theorem]{Remark}
\newtheorem{SA}[theorem]{Standing Assumption}

\newcommand{\of}{[\hspace{-0.06cm}[}
\newcommand{\gs}{]\hspace{-0.06cm}]}

\newcommand\llambda{{\mathchoice
		{\lambda\mkern-4.5mu{\raisebox{.4ex}{\scriptsize$\backslash$}}}
		{\lambda\mkern-4.83mu{\raisebox{.4ex}{\scriptsize$\backslash$}}}
		{\lambda\mkern-4.5mu{\raisebox{.2ex}{\footnotesize$\scriptscriptstyle\backslash$}}}
		{\lambda\mkern-5.0mu{\raisebox{.2ex}{\tiny$\scriptscriptstyle\backslash$}}}}}

\newcommand{\1}{\mathds{1}}

\newcommand{\F}{\mathbf{F}}
\newcommand{\G}{\mathbf{G}}

\newcommand{\B}{\mathbf{B}}

\newcommand{\M}{\mathcal{M}}

\newcommand{\la}{\langle}
\newcommand{\ra}{\rangle}

\newcommand{\lle}{\langle\hspace{-0.085cm}\langle}
\newcommand{\rre}{\rangle\hspace{-0.085cm}\rangle}
\newcommand{\blle}{\Big\langle\hspace{-0.155cm}\Big\langle}
\newcommand{\brre}{\Big\rangle\hspace{-0.155cm}\Big\rangle}

\newcommand{\X}{\mathsf{X}}

\newcommand{\tr}{\operatorname{tr}}
\newcommand{\N}{{\mathbb{N}}}
\newcommand{\cadlag}{c\`adl\`ag }
\newcommand{\on}{\operatorname}
\newcommand{\oP}{\overline{P}}
\newcommand{\oO}{\mathcal{O}}
\newcommand{\D}{D(\mathbb{R}_+; \mathbb{R})}
\newcommand{\bx}{\mathsf{x}}
\newcommand{\bb}{b} 
\newcommand{\bs}{\sigma} 
\newcommand{\bk}{k} 
\renewcommand{\o}{\diamond}
\newcommand{\ka}{\varkappa}
\newcommand{\ob}{\widetilde{b}}

\newcommand{\eps}{\epsilon}

\newcommand{\fPs}{\mathfrak{P}_{\textup{sem}}}
\newcommand{\fPas}{\mathfrak{P}^{\textup{ac}}_{\textup{sem}}}
\newcommand{\rrarrow}{\twoheadrightarrow}
\newcommand{\cA}{\mathcal{C}}
\newcommand{\cR}{\mathcal{R}}
\newcommand{\cK}{\mathcal{K}}
\newcommand{\cQ}{\mathcal{Q}}
\newcommand{\cF}{\mathcal{F}}
\newcommand{\cE}{\mathcal{E}}
\newcommand{\cC}{\mathcal{C}}
\newcommand{\cD}{\mathcal{D}}
\newcommand{\cP}{\mathcal{P}}
\newcommand{\bC}{\mathbb{C}}
\newcommand{\cH}{\mathcal{H}}
\newcommand{\bth}{\overset{\leftarrow}\theta}
\renewcommand{\th}{\theta}
\newcommand{\cG}{\mathcal{G}}

\newcommand{\bR}{\mathbb{R}}
\newcommand{\bN}{\mathbb{N}}
\newcommand{\nnabla}{\nabla}
\newcommand{\f}{\mathfrak{f}}
\newcommand{\g}{\mathfrak{g}}
\newcommand{\oconv}{\overline{\on{co}}\hspace{0.075cm}}
\renewcommand{\a}{\mathfrak{a}}
\renewcommand{\b}{\mathfrak{b}}
\renewcommand{\d}{\mathsf{d}}
\newcommand{\bS}{\mathbb{S}^d_+}
\newcommand{\p}{\dot{\partial}}
\newcommand{\dr}{r} 
\newcommand{\m}{\mathbb{M}}
\newcommand{\Q}{Q}
\newcommand{\C}{\mathsf{C}}
\newcommand{\n}{\overline{\nu}} 
\newcommand{\usc}{\textit{USC}}
\newcommand{\lsc}{\textit{LSC}}
\newcommand{\q}{\mathfrak{q}}

\renewcommand{\emptyset}{\varnothing}

\allowdisplaybreaks

\makeatletter
\@namedef{subjclassname@2020}{%
	\textup{2020} Mathematics Subject Classification}
\makeatother

 \title[Stochastic Representation of Sublinear Semigroups]{A Stochastic Representation Theorem for Sublinear Semigroups with non-local Generators} 
\author[D. Criens]{David Criens}
\author[L. Niemann]{Lars Niemann}
\address{Albert-Ludwigs University of Freiburg, Ernst-Zermelo-Str. 1, 79104 Freiburg, Germany}
\email{david.criens@stochastik.uni-freiburg.de}
\email{lars.niemann@stochastik.uni-freiburg.de}


\thanks{We are grateful to an anonymous referee of our previous paper \cite{CN22b} for advice concerning viscosity methods.
}

\thanks{LN acknowledges financial support from the DFG project SCHM 2160/13-1.}
\date{\today}

\maketitle

\begin{abstract}
In this paper we investigate sublinear semigroups whose pointwise generators are given by non-local Hamilton--Jacobi--Bellman operators. Our main result provides a stochastic representation in terms of a family of sublinear (conditional) expectations that can be understood as a nonlinear Markov family with uncertain local characteristics.
The proofs are based on viscosity methods. 
\end{abstract}

\section{Introduction}

Let \((T_t)_{t \in \mathbb{R}_+}\) be a linear conservative Feller semigroup on the Banach space \(C_0 (\bR^d; \bR)\) of continuous functions \(\bR^d \to \bR\) that are vanishing at infinity. It is classical that \((T_t)_{t \in \bR_+}\) can be uniquely characterized by its infinitesimal (pointwise) generator \(A\) and in applications it is standard to model the semigroup via a concrete version of its generator. Another well-known fact is that \((T_t)_{t \in \bR_+}\) has a stochastic representation of the form
\[
T_t (\psi) (x) = E^{P_x} \big[ \psi (X_t) \big], \quad t \in \bR_+, \ \psi \in C_0 (\bR^d; \bR), \ x \in \bR^d, 
\]
where \(\{P_x \colon x \in \bR^d\}\) is the associated Feller family (consisting of probability measures on a path space with coordinate process \(X\)). 
This representation is particularly useful when the dynamics of \(\{P_x \colon x \in \bR^d\}\) can be described in a tractable manner. 
For example, in case the generator \(A\) is a non-local operator of the type
\begin{equation} \label{eq: A intro} \begin{split}
	A (\psi) (x) = \langle \nabla & \psi (x), b (x) \rangle + \tfrac{1}{2} \on {tr} \big[ \nabla^2 \psi (x) \sigma \sigma^* (x) \big]
	\\&+ \int \big[ \psi (x + k (x, z)) - \psi (x) - \langle \nabla \psi (x), h (k (x, z)) \rangle \big] \, \n (dz),
\end{split} \end{equation}
for a suitable class of test functions \(\psi \colon \bR^d \to \bR\), then 
each measure \(P_x\) is the law of a solution process \((Y^x_t)_{t \in \bR_+}\) to the stochastic differential equation
\begin{alignat*}{2}
	d Y^x_t &= b (Y^x_t) dt + \sigma (Y^x_t) d W_t &&+ \int \, (k (Y^x_t, y) - h (k (Y^x_t, y))) (\mu^L - \nu^L) (dt, dy) \\&&&+ \int \, h (k (Y^x_t, y)) \mu^L (dt, dy),
\end{alignat*}
with initial condition \(Y^x_0 = x\), where \(W\) is a standard Brownian motion and \(\mu^L\) is a Poisson random measure with intensity measure~\(\nu^L = dt \otimes \n\). This structure of the generator is quite typical in the sense that, if the space \(C^\infty_c(\bR^d; \bR)\)  of smooth functions with compact support is contained in the domain of \(A\), the celebrated Courr\`ege--von Waldenfels theorem (see, e.g., \cite[Theorem~2.21]{schilling}) shows that the pointwise generator \(A\) is of the form \eqref{eq: A intro} for test functions \(\psi\) from \(C^\infty_c (\bR^d; \bR)\). 

\subsubsection*{The purpose of this article}
In this paper, we are interested in versions of these results for {\em sublinear} semigroups, i.e., families \((T_t)_{t \in \bR_+}\) of \emph{sublinear} operators. 
More specifically, we consider a (conservative) semigroup \((T_t)_{t \in \bR}\) of sublinear operators on the space \(C_b (\bR^d; \bR)\) of bounded continuous functions whose pointwise generator \(A\) is given as a non-local Hamilton--Jacobi--Bellman operator of the form
\begin{equation} \label{eq: gen intro} \begin{split}
A (\psi) (x) = & \, \sup  \Big\{ \langle \nabla \psi (x), b (f,x) \rangle + \tfrac{1}{2} \on {tr} \big[ \nabla^2 \psi (x) \sigma \sigma^* (f,x) \big]
\\& \ \, + \int \big[ \psi (x + k (f, x, z)) - \psi (x) - \langle \nabla \psi (x), h (k (f, x, z)) \rangle \big] \, \n (f, dz) \colon f \in F \Big\}
\end{split}
\end{equation}
for some (arbitrary) parameter space \(F\) and three functions \(b \colon F \times \bR^d \to \bR^d\), \(\sigma \colon F \times \bR^d \to \mathbb{S}^d_+\) and
\(k \colon F \times \bR^d \times L \to \bR^d \), where \(L\) is a Lusin space equipped with a family \(\mathscr{N} := \{\n\, (f, dz) \colon f \in F\}\) of L\'evy-type measures.

As in the linear case, this structure of \(A\) is somehow typical. 
Indeed, a version of the Courr\`ege--von Waldenfels theorem for sublinear semigroups has recently been proved in \cite{K21}. It shows that for any conservative sublinear Markovian semigroup on \(C_b(\bR^d; \bR)\), for which \(C^\infty_c(\bR^d;\bR)\) is contained in the domain of the pointwise generator, latter is of form \eqref{eq: gen intro} for test functions \(\psi\) from \(C^\infty_c(\bR^d; \bR)\).

We ask for conditions on \(b, \sigma, k\) and \(\mathscr{N}\) that ensure a tractable stochastic representation that is useful to study properties of \((T_t)_{t \in \bR_+}\).
In particular, a stochastic representation only depending on the input data \(b, \sigma, k\) and \(\mathscr{N}\) shows that the pointwise generator given in \eqref{eq: gen intro} characterizes the semigroup \((T_t)_{t \in \bR_+}\) uniquely.

\subsubsection*{The main result}
A natural replacement for the Feller family \(\{P_x \colon x \in \bR^d\}\) from the linear case is a nonlinear Markov family \(\{\cE^x \colon x \in \bR^d\}\) consisting of sublinear expectations of the form
\[\cE^x (\, \cdot \,) = \sup_{P \in \cR(x)} E^P \big[ \cdot \big]\] with a collection \(\cR(x)\) of semimartingale laws with initial distribution \(\delta_x\) whose local semimartingale characteristics are prescribed in a Markovian way.
More precisely, we set
\begin{align*}
\cR(x) := \Big\{P \in \mathfrak{P}_{\text{sem}}^{\text{ac}} \colon P \circ X_0^{-1} = \delta_x,  \,
(dt \otimes P)\text{-a.e. } (dB^P /d t, dC^P/d t, \nu^P / d t) \in \Theta (X) \Big\},
\end{align*}
where
\[
\Theta (x) := \big\{(b (f, x), \sigma \sigma^* (f, x), \n (f, \, \cdot\,) \o k (f, x, \cdot \,)^{-1}) \colon f \in F \big\}, \quad x \in \bR^d.
\]
The main result of this paper provides conditions on the coefficients \(b, \sigma, k\) and \(\mathscr{N}\) such that \((T_t)_{t \in \bR_+}\) has the stochastic representation 
\[
T_t (\psi) (x) = \mathcal{E}^x (\psi (X_t)) = \sup_{P \in \cR (x)} E^P \big[ \psi (X_t) \big]
\]
for all \(t \in \bR_+\), \(\psi \in C_b (\bR^d; \bR)\) and \(x \in \bR^d\). In particular, our result establishes a one-to-one correspondence between the sublinear semigroup and its (pointwise) generator. 

\subsubsection*{Comments on the proof}
A general representation result for certain sublinear Markovian semigroups on \(C_b(\bR^d; \bR)\) in terms of an abstract nonlinear Markov process was proved in \cite{denk2018extension} and later applied in \cite{denk2020semigroup} to sublinear Markovian convolution semigroups.
For the latter, explicit stochastic representations have been obtained and investigated in terms of nonlinear L\'evy processes as introduced in \cite{neufeld2017nonlinear}, see \cite{hol16, K19, K21}. 
Furthermore, for such semigroups, the paper \cite{K21} shows a one-to-one correspondence between the semigroup and its pointwise generator. 
Certain stochastic representations for nonlinear (jump-)diffusion frameworks related to this paper, albeit under different assumptions, have been obtained in the recent papers \cite{CN22a, CN22b, CN23a}. 

A central tool to derive the representation theorem is 
the link of the semigroup \((T_t)_{t \in \bR_+}\) and its generator \(A\) through an evolution equation.
More precisely, in case the domain of the (pointwise) generator is rich enough, the map \((t,x) \mapsto T_t(\psi)(x)\) solves the evolution equation
\begin{equation} \label{eq: PDE A intro}
\begin{cases}   
\partial_t u (t, x) - A(u(t, \cdot \,))(x) = 0, & \text{for } (t, x) \in (0, \infty) \times \mathbb{R}^d, \\
u (0, x) = \psi (x), & \text{for } x \in \bR^d,
\end{cases}
\end{equation}
in a viscosity sense, provided it is continuous. 
The stochastic representation and the unique characterization via the generator follows once one proves that the map 
\[
v (t, x) := \mathcal{E}^x (\psi (X_t)), \quad (t, x) \in \bR_+ \times \bR^d,
\]
is the unique viscosity solution to the evolution equation \eqref{eq: PDE A intro}. At this point, we highlight that the idea to use viscosity methods to get a representation of sublinear semigroups traces back to the paper \cite{lions_nisio}.

Previous approaches (as, e.g., in \cite{CN22b,CN23a,hol16,K19,K21,neufeld2017nonlinear}) to identify \(v\) as the unique viscosity solution to \eqref{eq: PDE A intro} are usually divided into three parts. First, one shows that the map \(v\) is continuous, which is closely related to the \(C_b\)--Feller property of the sublinear semigroup \(S_t (\psi) (x) := \mathcal{E}^x (\psi (X_t))\), and second, one proves the mere sub- and supersolution properties. Finally, a comparison principle for viscosity solutions completes the proof.

In this paper, we address the issue all at once, adapting a suggestion of a referee of our paper \cite{CN22b}. We verify that the upper and lower envelopes of \(v\) are viscosity sub- and supersolutions. Thanks to a general comparison result for viscosity solutions to non-local Hamilton--Jacobi--Bellman PDEs from \cite{hol16}, we then conclude that \(v\) is a viscosity solution to~\eqref{eq: PDE A intro}. In particular, this verifies the continuity of \(v\) en passant. 

We highlight that, compared to the representation theorem from \cite{CN23a}, we demand fewer assumptions regarding the modeling structure. Specifically, the action space \(F\) is not assumed to be compact, the coefficients can be unbounded, less integrability conditions are imposed and
\(\Theta (x), \, x \in \bR^d,\) is not required to be convex.
This last point is of particular significance from a modeling perspective, as it permits the direct treatment of the coefficient \(b, \sigma, k\) and \(\mathscr{N}\) without the need to consider the convexity of \(\Theta\).
Observe, however, that the application of the comparison principle from \cite{hol16} requires global Lipschitz continuity of the coefficients. In contrast, the paper \cite{CN23a} establishes the \(C_b\)--Feller property of \((S_t)_{t \in \bR_+}\) under {\em local} Lipschitz conditions.

\subsubsection*{Further comments on related literature}
The investigation of nonlinear stochastic processes dates back to the  seminal work of Peng \cite{peng2007g, peng2008multi} on the \(G\)-Brownian motion. Leaning on ideas developed in \cite{neufeld2014measurability, neufeld2017nonlinear, NVH}, a general framework for constructing nonlinear Markov processes was established in \cite{hol16}.
In recent years, several classes of nonlinear (Markov) processes were constructed and analyzed, both from the perspective of processes under uncertainty \cite{biagini, C22a, CN22a, CN22b, CN23a, fadina2019affine, hol16, hu2021g,  neufeld2017nonlinear, nutz}, as well as 
sublinear Markovian semigroups \cite{denk2020semigroup, GNR22, K19, K21, NR}. We refer to \cite[Chapter 4]{hol16} and \cite{CN22b} for a comparison of the methods.

Let us also mention that an analytic approach to the unique characterization of sublinear semigroups not relying on viscosity methods was recently proposed in \cite{blessing22}. 
More specifically, the paper provides a comparison principle for sublinear semigroups (with certain properties) in terms of so-called \(\Gamma\)-generators. To illustrate the theory, the authors apply their result, for example, to connect the Nisio semigroup, as constructed in \cite{NR}, to the control semigroup of \(G\)-Brownian motion. 
It is an interesting open problem to investigate how this theory relates to our main results. 
We leave this question for future investigations.

\subsubsection*{Structure of the paper}
This paper is structured as follows: in Section \ref{sec: nonlinear markov} we introduce our setting and in Section~\ref{sec: main results} we present our main results on the unique characterization of the semigroup in terms of its pointwise generator.
The proofs of our main results are given in Section \ref{sec: proofs}.
We collect auxiliary results for L\'evy-type HJB equations in the Appendices~\ref{app: comparison} and \ref{app: well-defined}.

\section{Nonlinear Markov Processes with Jumps} \label{sec: nonlinear markov}

In this section we introduce the mathematical setting.
Our main results are presented in Section \ref{sec: main results}.

\subsection{The Framework}
Let \(d \in \mathbb{N}\) be a fixed dimension and define $\Omega$ to be the space of all \cadlag functions \(\mathbb{R}_+ \to \mathbb{R}^d\) endowed with the Skorokhod \(J_1\) topology. The Euclidean scalar product and the corresponding Euclidean norm are denoted by \(\langle\, \cdot\, , \cdot\,\rangle\) and \(\|\cdot\|\), respectively.
Further, we write \(X\) for the canonical process on $\Omega$, i.e., \(X_t (\omega) = \omega (t)\) for \(\omega \in \Omega\) and \(t \in \mathbb{R}_+\). 
It is well-known that \(\mathcal{F} := \mathcal{B}(\Omega) = \sigma (X_t, t \geq 0)\).
We define $\F := (\mathcal{F}_t)_{t \geq 0}$ as the canonical filtration generated by $X$, i.e., \(\mathcal{F}_t := \sigma (X_s, s \leq t)\) for \(t \in \mathbb{R}_+\). 

The set of probability measures on \((\Omega, \mathcal{F})\) is denoted by \(\mathfrak{P}(\Omega)\) and endowed with the usual topology of convergence in distribution.
Let \(\mathbb{S}^d_+\) be the set of all symmetric nonnegative semidefinite real-valued \(d \times d\) matrices, and set 
\[
\mathcal{L} := \Big\{ K \text{ measure on } (\bR^d, \mathcal{B}(\bR^d)) \colon K (\{0\}) = 0, \  \int \,(\|x\|^2 \wedge 1)\, K (dx) < \infty\Big\}.
\]
We endow \(\mathcal{L}\) with the weakest topology under that the maps
\[
K \mapsto \int f (x) (\|x\|^2 \wedge 1)\, K (dx), \quad f \in C_b (\bR^d; \bR), 
\]
are continuous. With this topology, the space \(\mathcal{L}\) is Polish, see \cite[Lemma~A.1]{C22a}. 

Take an arbitrary set \(F\), a dimension \(r \in \mathbb{N}\) and a Lusin space \(L\). For every \(f \in F\), let \(\n (f, dz)\) be a (possibly infinite) measure on \((L, \mathcal{B}(L))\). 
Further, fix three functions \begin{align*}
	&b \colon F \times \bR^d \to \bR^d, \\
	&\sigma \colon F \times \bR^d \to \bR^{d \times r}, \\
	&k \colon F \times \bR^d \times L \to \bR^d
\end{align*}
such that, for all \((f, y) \in F\times \bR^d\), \(z \mapsto k(f, y, z)\) is Borel measurable and 
\begin{align} \label{eq: assumption k inte}
\int \,(\|k (f, y, z)\|^2 \wedge 1)\, \n (f, dz) < \infty.
\end{align}
For a Borel function \(g \colon L \to \mathbb{R}^d\) and a measure \(\nu\) on \((L, \mathcal{B}(L))\), we define \(\nu \o g^{-1}\) to be the measure
\[
\nu \o g^{-1} (G) := \int \1_{G \backslash \{0\}} (g (z))\, \nu (dz), \quad G \in \mathcal{B}(\bR^d). 
\]
Finally, we define a correspondence, i.e., a set-valued mapping, \(\Theta \colon \bR^d \twoheadrightarrow \mathbb{R}^d \times \mathbb{S}^d_+ \times \mathcal{L}\) by
\[
\Theta (x) := \big\{ \big(b (f,x), \sigma \sigma^* (f,x), \n (f, \, \cdot\, ) \o k (f, x, \cdot\, )^{-1} \big) \colon f \in F \big\}, \quad x \in \bR^d.
\]
Below, we use the correspondence \(\Theta\) to incorporate ambiguity to the infinitesimal description of a Markov process.
In this paper, we work under the following standing assumption. 

\begin{SA} \label{SA: MG}
\(\Theta\) has a measurable graph, i.e., the set
\[
\on{gr} \Theta = \big\{ (x, b, a, K) \in \bR^d \times \bR^d \times \mathbb{S}^d_+ \times \mathcal{L} \colon (b, a, K) \in \Theta (x) \big\} 
\]
is Borel. 
\end{SA} 

\begin{remark}
    Standing Assumption \ref{SA: MG} holds, for example, under the following condition (cf. \cite[Theorem~8.2.8]{SVA_AF}):
    \begin{enumerate}
        \item[\textup{(MG)}]
        \(F\) is a closed subset of a Polish space, \(b, \sigma\) and \(\n\o k^{-1}\) are Carath\'eodory functions (in the sense that they are continuous in the \(F\) variable and Borel in the remaining variables), and \(\Theta\) is closed-valued, i.e., \(\Theta (x)\) is closed for every \(x \in \bR^d\).
    \end{enumerate}
    In particular, notice that the correspondence \(\Theta\) is closed-valued once \(b, \sigma\) and \(\n \o k^{-1}\) are Carath\'eodory functions and \(F\) is a compact space. 
\end{remark}

\subsection{Nonlinear Markov Families and their Semigroups}

Let \(\fPs\) be the set of all semimartingale measures, i.e., the set of all laws of \(\bR^d\)-valued semimartingales. For \(P \in \fPs\) we denote the corresponding semimartingale characteristics, with respect to a fixed Lipschitz continuous truncation function \(h \colon \bR^d \to \bR^d\), by \((B^P, C^P, \nu^P)\).
Moreover, we set 
\[
\fPas:= \big\{ P \in \fPs \colon P\text{-a.s. } (B^P, C^P, \nu^P) \ll \llambda \big\}, 
\]
which is the set of all semimartingale measures with absolutely continuous semimartingale characteristics (with respect to the Lebesgue measure \(\llambda\)). 
Finally, for $x \in \bR^d$, we define 
\begin{align*}
\cR (x) := \Big\{ P \in \fPas \colon P(X_0 = x) = 1, \,
(\llambda \otimes P)\text{-a.e. } (dB^P /d\llambda, dC^P/d\llambda, \nu^P/ d \llambda) \in \Theta (X) \Big\}.
\end{align*}
From now on, we work under the following standing assumption.
\begin{SA} \label{SA: non empty}
\(\cR (x) \not = \emptyset\) for all \(x \in \bR^d\).
\end{SA}

\begin{remark} \label{rem: non empty}
    Standing Assumption \ref{SA: non empty} can be deduced from existence theorems for Markovian semimartingale laws as given in \cite{criens19,EbKa19,jacod79,JS}.
    For instance, Proposition \ref{prop: existence} below shows that it holds under appropriate boundedness and continuity assumptions on \(b, \sigma, k\) and \(\n\), see also Remark \ref{rem: existence} below.
\end{remark}

Recall that a function \(\psi \colon \Omega \to [- \infty, \infty]\) is called upper semianalytic if the set \(\{ \omega \in \Omega \colon \psi (\omega) > c\}\) is analytic for every \(c \in \mathbb{R}\). We define a family \(\{\cE^x \colon x \in \bR^d\}\) of sublinear expectations on the cone of upper semianalytic functions by the formula
 \[
 \mathcal{E}^x (\psi) := \sup_{P \in \cR (x)} E^{P} \big[ \psi \big]
 \]
 for \(x \in \bR^d\) and upper semianalytic \(\psi \colon \Omega \to [- \infty, \infty]\). As we explain next, the family \(\{\mathcal{E}^x \colon x \in \bR^d\}\) can be interpreted as a nonlinear Markov family. 
 
 For \(t \in \mathbb{R}_+\), we denote the shift operator \(\theta_t \colon \Omega \to \Omega\) by \[\theta_t (\omega) := \omega(\,\cdot + t), \quad \omega \in \Omega.\]
The next result provides the so-called nonlinear Markov property of the family
\( \{ \cE^x \colon x \in \bR^d\}\), cf. \cite[Lemma 4.32]{hol16} and \cite[Proposition 2.8]{CN22b}. We omit a detailed proof.

\begin{proposition}  [Nonlinear Markov Property] \label{prop: markov property}
For every upper semianalytic function \( \psi \colon \Omega \to [- \infty, \infty] \), the equality
\[
\cE^x( \psi \circ \theta_t) = \cE^x ( \cE^{X_t} (\psi))
\]
holds for every \((t, x) \in \bR_+ \times \bR^d\).
\end{proposition}

Proposition \ref{prop: markov property} confirms the intuition that \( \{\cE^x \colon x \in \bR^d\} \) is a nonlinear Markov family, as it implies the equality
\[ \cE^x(\psi(X_{s+t})) = \cE^x( \cE^{X_t}( \psi (X_s))) \]
for every upper semianalytic function \(\psi \colon \bR^d \to [- \infty, \infty]\), \(s,t \in \bR_+ \) and \(x \in \bR^d\).

In the linear case, Markov families have an intrinsic relation to operator semigroups. As we discuss next, the same is true for nonlinear Markov families. 

\begin{definition}
Let \( \mathcal{H} \) be a convex cone of functions \( f \colon \bR^d \to \bR \) containing all constant functions.
A family of sublinear operators \( S_t \colon \mathcal{H} \to \mathcal{H}, \ t \in \bR_+,\) is called a \emph{sublinear Markovian semigroup} on \( \mathcal{H} \) if it satisfies the following properties:
\begin{enumerate}
    \item[\textup{(i)}] \( (S_t)_{t \in \bR_+} \) has the semigroup property, i.e., 
          \( S_s S_t = S_{s+t} \) for all \(s, t \in \bR_+ \)  and
          \( S_0 = \on{id} \),
    
    \item[\textup{(ii)}] \( S_t \) is monotone for each \( t \in \bR_+\), i.e., 
    \( f, g \in \mathcal{H} \) with \( f \leq g \) implies \(S_t (f) \leq S_t (g) \),
    
    \item[\textup{(iii)}] \( S_t \) preserves constants for each  \( t \in \bR_+\), i.e.,
    \( S_t(c) = c \) for each \( c \in \bR  \).
\end{enumerate}
\end{definition}

We introduce a family of operators \( (S_t)_{t \in \bR_+} \) by the formula
\begin{equation} \label{eq: semigroup}
    S_t ( \psi )(x) := \cE^x(\psi(X_t)), \quad (t, x) \in \bR_+ \times \bR^d.
\end{equation}

The following result is a consequence of Proposition~\ref{prop: markov property}. It is a restatement of \cite[Remark~4.33]{hol16} for our framework.

\begin{proposition} [Nonlinear Semigroup Property]
The family of operators \( (S_t)_{t \in \bR_+} \) from \eqref{eq: semigroup}
defines a sublinear Markovian semigroup on the set of bounded upper semianalytic functions from \(\bR^d\) to \(\bR\).
\end{proposition}

It is well-known that linear semigroups on spaces of sufficiently regular functions can be uniquely characterized by their infinitesimal (pointwise) generators and that they admit a stochastic representation via a Markov family. In the following section, we provide a similar result for nonlinear Markov semigroups.

\section{The Stochastic Representation Theorem} \label{sec: main results}

The goal of this section is to show that \((S_t)_{t \in \bR_+}\) is the only jointly continuous\footnote{Here, jointly continuous means that \((t, x) \mapsto S_t (\psi) (x)\) is continuous for all \(\psi \in C_b(\bR^d;\bR)\).} sublinear semigroup on \(C_b(\bR^d; \bR)\) whose pointwise generator is given by the formula 
\begin{equation} \begin{split} \label{eq: G}
    G(x, \phi) := &\, \sup  \Big\{ \langle \nabla \phi (x), b (f, x) \rangle
    + \tfrac{1}{2} \on{tr} \big[ \nabla^2\phi (x) \sigma \sigma^* (f, x) \big]
    \\& \ \, + \int\, \big[ \phi(x + \bk (f, x, z)) - \phi(x) - \langle \nabla \phi(x) , h (k (f, x, z)) \rangle\big]\, \n (f, dz)
    \colon f \in F \Big\}.
\end{split}
\end{equation}
The proof of this result is based on viscosity methods. Namely, under suitable conditions, we identify the map \((t, x) \mapsto S_t (\psi) (x)\) as the unique viscosity solution to the equation 
\begin{equation} \label{eq: G eq}
\begin{cases}   
\partial_t u (t, x) - G (x, u (t, \, \cdot\, )) = 0, & \text{for } (t, x) \in (0, \infty) \times \mathbb{R}^d, \\
u (0, x) = \psi (x), & \text{for } x \in \bR^d.
\end{cases}
\end{equation}
In particular, we establish the continuity of \((t, x) \mapsto S_t (\psi) (x)\) for any \(\psi \in C_b (\bR^d; \bR)\), which entails the \(C_b\)--Feller property of \((S_t)_{t \in \bR_+}\), i.e., \(S_t (C_b (\bR^d;\bR) ) \subset C_b (\bR^d; \bR)\) for any \(t \in \bR_+\).

For any jointly continuous semigroup \((T_t)_{t \in \bR_+}\) with pointwise generator \(G\), we can adapt results from \cite{hol16} to show that \((t, x) \mapsto T_t (\psi) (x)\) solves \eqref{eq: G eq} in a viscosity sense. In turn, the uniqueness result implies that \(S_t (\psi) (x) = T_t (\psi) (x)\), which is the desired representation of the semigroup.

\smallskip
The remainder of this section is split into two parts. First, we establish the relation of \((S_t)_{t \in \bR_+}\) to the equation \eqref{eq: G eq} and second, we deduce the stochastic representation for semigroups with pointwise generator \(G\).

\subsection{The Generator Equation and the \(C_b\)--Feller Property}
Recall that a function \(u \colon \bR_+ \times \bR^d \to \mathbb{R}\) is said to be a \emph{viscosity subsolution} to \eqref{eq: G eq} if it is upper semicontinuous and the following two properties hold:
\begin{enumerate}
    \item[\textup{(a)}] \(u(0, \cdot\, ) \leq \psi\);
\item[\textup{(b)}]
\(
\partial_t \phi (t, x) - G (x, u(t, \cdot\,)) \leq 0
\)
for all \(\phi \in  C^\infty_b(\bR_+ \times \bR^d; \bR)\) such that \(u \leq \phi\) and \(\phi (t, x) = u(t, x)\) for some \((t, x) \in (0, \infty) \times \bR^d \). 
\end{enumerate}
A \emph{viscosity supersolution} is obtained by replacing upper with lower semicontinuity and reversing the inequalities. Further, \(u\) is called \emph{viscosity solution} if it is both, a viscosity sub- and supersolution. 

\smallskip
For the remainder of this section, fix \(\psi \in C_b (\bR^d; \bR)\) and set 
\[
v (t, x) := S_t (\psi )(x), \quad (t, x) \in \bR_+ \times \bR^d.
\]
We consider the upper and lower envelopes of \(v\), i.e., for \((t, x) \in \bR_+ \times \bR^d\),
\[
v^* (t, x) := \limsup_{(s, y) \to (t, x)} v (s, y), \qquad v_* (t, x) := \liminf_{(s, y) \to (t, x)} v (s, y).
\]
Notice that \(v^*\) is upper semicontinuous and that \(v_*\) is lower semicontinuous. 

\smallskip
Before we can state our main theorem, we formulate its prerequisites. 
To this end, let us define the \emph{symbol} associated to the coefficients \((b, \sigma\sigma^*, \n \o k^{-1})\) by 
\begin{equation} \label{eq: df symbol}
\begin{split}
\q(f, x, \xi)  := -i \langle b&(f,x), \xi \rangle + \frac{1}{2} \langle \xi, \sigma \sigma^*(f,x) \xi \rangle \\ &+ 
\int \big(1 - e^{i \langle y, \xi \rangle} + i \langle \xi, h(y) \rangle \big) (\n (f, \, \cdot\,) \o k(f,x, \cdot\,)^{-1})(dy),
\end{split}
\end{equation}
for \(x, \xi \in \bR^d\) and \(f \in F\). Here, \(i\) denotes the imaginary number.

\begin{condition} \label{cond: V1}
For every compact set \(K \subset \bR^d\),  
\begin{enumerate}
        \item[\textup{(i)}]
        \begin{align*}
           \sup \Big\{ \| b (f, x) \| + \|\sigma (f, x) \| \colon (f, x) \in F \times K \Big\}  < \infty,
          \end{align*}

        \item[\textup{(ii)}] the symbol \(\q\) is uniformly continuous, i.e.,
         \[
         \lim_{r \to \infty} \sup_{f \in F} \sup_{\| x\| \leq r} \sup_{\| \xi \| \leq r^{-1}}    | \q(f,x,\xi) | =0,
         \]

        \item[\textup{(iii)}] and there exists a Borel function \(\gamma = \gamma\, (K) \colon L \to \bR_+\) such that   \[\sup_{f \in F} \int (\gamma^2(z) \wedge 1)\, \n (f, dz) < \infty\] 
  and 
\begin{align*}
    \lim_{\kappa \searrow 0} \, \sup_{f \in F} \int_{\{\gamma\leq \kappa\}} \gamma^2 (z)\, \n (f, dz) &= 0, \quad
    \lim_{R \to \infty} \, \sup_{f \in F} \n \, (f, \{\gamma > R\}) = 0,
\end{align*} and a modulus of continuity\footnote{A modulus of continuity is an increasing function that is vanishing and continuous at zero.} \(\ka = \ka\, (K) \colon [0, \infty] \to [0, \infty]\) such that 
      \begin{align*}
          \|b (f, x) - b (f, y)\| + \|\sigma \sigma^* (f, x) - \sigma \sigma^* (f, y) \| &\leq \ka\, (\|x - y\|), \\
          \|k (f, x, z) - k (f, y, z)\| &\leq \gamma (z)\, \ka\, (\|x - y\|),
          \\ \|k (f, x, z) \| &\leq \gamma (z)
      \end{align*}
      for all \(f \in F\), \(x, y \in K\) and \(z \in L\).
\end{enumerate}
\end{condition}

\begin{remark}[Discussion of Condition \ref{cond: V1}] \quad
\begin{itemize} 
\item[\textup{(i)}] Notice that part (iii) of Condition \ref{cond: V1} yields that 
    \begin{align*}
        \sup \Big\{ \int ( \|k (f, x, z)\|^2 \wedge 1 ) \, \n (f, dz) \colon (f, x) \in F \times K \Big\}  < \infty,
    \end{align*}
    for every compact set \(K \subset \bR^d\).
    
\item[\textup{(ii)}] 
Let us comment on part (ii) of Condition \ref{cond: V1}. 
Note that it implies the following seemingly stronger condition 
\[
\lim_{r \to \infty} \sup_{f \in F}  \sup_{\| x\| \leq \lambda r} \sup_{\| \xi \| \leq r^{-1}} | \q(f,x,\xi) | =0, \quad \lambda \geq 1.
\]
To see this, recall that for fixed \((f,x) \in F \times \bR^d\) the function \( \xi \mapsto \sqrt{|\q(f,x,\xi)|}\) is subadditive (\cite[Proposition 2.17]{schilling}).
Part (ii) of Condition \ref{cond: V1} holds, for instance, in case the coefficients \((b, \sigma \sigma^*, \n \o k^{-1})\) satisfy the following two conditions:
\begin{enumerate}
    \item[\textup{(a)}] the coefficients are globally bounded in the sense that
\[
\sup \Big\{ \| b (f, x) \| + \|\sigma (f, x) \| + \int ( \|k (f, x, z)\|^2 \wedge 1 ) \, \n (f, dz) \colon (f, x) \in F \times \bR^d \Big\}  < \infty;
\]
\item[\textup{(b)}] the coefficients satisfy the tightness condition 
\[
\lim_{R \to \infty} \, \sup \big\{ \n\, ( f, \{\| k (f, x, \, \cdot\, )\| > R\} ) \colon (f, x) \in F \times \bR^d \big\} = 0.
\]
\end{enumerate}
Indeed, under these assumptions, Condition \ref{cond: V1} (ii) follows from \cite[Lemma~A.2]{K19}.
\end{itemize}
\end{remark}

The following result shows that all semimartingale problems for constant controls can be solved if Condition \ref{cond: V1} is in force. In particular, it implies Standing Assumption~\ref{SA: non empty}.

\begin{proposition} \label{prop: existence}
     Suppose that Condition \ref{cond: V1} holds. Then, for every \(f \in F\) and \(x \in \bR^d\), there exists a semimartingale law \(P \in \fPas\) such that \(P (X_0 = x) = 1\) and \((\llambda \otimes P)\)-a.e. 
    \[
    \Big( \frac{d B^P}{d \llambda}, \frac{d C^P}{d \llambda}, \frac{d \nu^P}{d \llambda} \Big) = \big(b (f, X), \sigma \sigma^* (f, X), \n (f, \, \cdot\, ) \o k (f, X, \, \cdot \,)^{-1} \big).
    \]
\end{proposition}
\begin{proof}
Fix \(f \in F\) and set \(\q_f(x, \xi) := \q(f, x, \xi)\) for \((x,\xi) \in \bR^d \times \bR^d\). This is the symbol associated to the constant control \(f\). Observe that under Condition \ref{cond: V1} the following holds:
\begin{itemize} \item[\textup{(i)}] the symbol \(\q_f\) has continuous coefficients, i.e.,
    the maps \(x \mapsto b(f,x)\) and \(x \mapsto \sigma(f,x)\) are continuous, and for every \(i,j =1,\dots,d\) and every continuous bounded function \(g \colon \bR^d \to \bR\) vanishing around the origin, the functions
        \begin{equation} \label{eq: continuity test}
         x \mapsto \int h^i(k(f,x,z)) h^j(k(f,x,z)) \, \n(f, dz), \quad x \mapsto \int g(k(f,x,z)) \, \n(f, dz)   
        \end{equation}
        are continuous, too;

    \item[\textup{(ii)}] and, as a consequence, the symbol \(\q_f\) is locally bounded, i.e., for every compact set \(K \subset \bR^d\),  
    \[
    \sup_{x \in K} \Big\{ \| b (f, x) \| + \|\sigma (f, x) \| +  \int ( \|k (f, x, z)\|^2 \wedge 1 ) \, \n (f, dz) \Big\}  < \infty;
    \]

\end{itemize}
Indeed, (i) follows from part (iii) of Condition \ref{cond: V1}. Specifically, continuity in \eqref{eq: continuity test} is a consequence of part (iii) of Condition \ref{cond: V1} and the dominated convergence theorem.
Of course, (ii) can also be deduced from (i) and (iii) of Condition \ref{cond: V1} without referring to the continuity of the coefficients of \(\q_f\).
Hence, as the symbol \(\q_f\) is uniformly continuous (in the sense of \cite{franzi}) by part (ii) of Condition \ref{cond: V1}, an application of \cite[Theorem 1.1]{franzi} in conjunction with \cite[Proposition 4.31]{EbKa19} completes the proof.
\end{proof}

\begin{remark} \label{rem: existence}
Let us mention that Proposition \ref{prop: existence} can be established under more general conditions. Using the same notation as in the proof of Proposition~\ref{prop: existence}, if the symbol \(\q_f\) has continuous coefficients in the sense of part (i) in the proof of Proposition \ref{prop: existence}, and additionally, the growth condition
        \[
        \limsup_{\|x\| \to \infty} \sup_{\|\xi \| \leq 1/ \|x\|} |\q_f(x, \xi)| < \infty,
        \]
        holds, then the conclusion of Proposition \ref{prop: existence} follows from \cite[Theorem 5.36]{EbKa19}.
     
\end{remark}

Finally, we state a condition that allows us to use a comparison theorem for viscosity solutions to Hamilton--Jacobi--Bellman PDEs. 
We emphasize that it implies Condition~\ref{cond: V1} above, see Lemma~\ref{lem: lemma cont modi drift} in Appendix~\ref{app: well-defined} for details.
\begin{condition} \label{cond: V3}
    Let \(\gamma \colon L \to \bR_+\) be a Borel map such that  \[\sup_{f \in F} \int (\gamma^2(z) \wedge 1)\, \n (f, dz) < \infty\] 
  and 
\begin{align*}
    \lim_{\kappa \searrow 0} \, \sup_{f \in F} \int_{\{\gamma\leq \kappa \}} \gamma^2 (z)\, \n (f, dz) &= 0, \quad
    \lim_{R \to \infty} \, \sup_{f \in F} \n \, (f, \{\gamma > R\}) = 0,
\end{align*} 
and define the modified drift coefficient \(\ob \colon F \times \bR^d \to \bR^d\) by
    \begin{equation*} \begin{split}
    \ob (f, x) := b (f, x) &+ \int_{\{\gamma \leq 1\}} \big( k (f, x, z) - h (k (f, x, z)) \big)\, \n (f, dz) \\&- \int_{\{\gamma > 1\}} h (k (f, x, z)) \, \n (f, dz).
    \end{split}
    \end{equation*}
    We formulate the following conditions, which also entail that \(\ob\) is well-defined (see Appendix~\ref{app: well-defined} for details). 
    
    \begin{enumerate}
        \item[\textup{(i)}] For every compact set \(K \subset \bR^d\), 
        \begin{align*}
           \sup \Big\{ \| b (f, x) \| + \|\sigma (f, x) \| \colon (f, x) \in F \times K \Big\}  < \infty.
        \end{align*}
        \item[\textup{(ii)}] The symbol \(q\) is uniformly continuous, i.e.,
         \[
         \lim_{r \to \infty} \sup_{f \in F} \sup_{\| x\| \leq r} \sup_{\| \xi \| \leq r^{-1}} | \q(f,x,\xi) | =0.
         \]     
         \item[\textup{(iii)}]
          There exists a constant \(C > 0\) such that 
          \begin{align*}
              \|\ob (f, x) - \ob (f, y) \| + \|\sigma (f, x) - \sigma (f, y) \| &\leq C \|x - y\|, \\
              \|k (f, x, z) - k (f, y, z) \| &\leq \gamma (z) \|x - y\|,
              \\\|k (f, x, z) \| &\leq \gamma (z) ( 1 + \|x\| ) 
          \end{align*}
          for all \(f \in F, x, y \in \bR^d\) and \(z \in L\).
    \end{enumerate}
\end{condition}

\begin{remark}
    In case \(b\) and \(k\) are such that
              \begin{align*}
              \| b(f, x) - b (f, y) \| &\leq C \|x - y\|, \\
              \|k (f, x, z) - k (f, y, z) \| &\leq \gamma (z) \|x - y\|
          \end{align*}
          for all \(f \in F, x, y \in \bR^d\) and \(z \in L\), then \(\ob\) satisfies its Lipschitz condition from Condition~\ref{cond: V3}~(iii) under each of the following conditions:
          \begin{enumerate}
              \item[\textup{(a)}] \(\sup_{f \in F} \int \gamma (z) \, \n (f, dz) < \infty\);
              \item[\textup{(b)}] \(\sup_{f \in F}\int (\gamma^2 (z) \wedge \gamma (z)) \, \n (f, dz) < \infty\) and 
              \[
              \|k (f, x, z)\| \leq \gamma (z) 
              \]
              for all \(f \in F, x \in \bR^d\) and \(z \in L\). 
          \end{enumerate}
        Here, it is easy to see that (a) suffices. To understand that \(\ob\) is Lipschitz continuous under hypothesis (b), notice that 
        \begin{equation} \label{eq: small gamma weg}
        \begin{split}
            \int_{\{\gamma \leq 1\}} \big( k (f, x, z) &- h (k (f, x, z)) \big)\, \n (f, dz) 
            \\&= \int_{\{\kappa < \gamma \leq 1\}} \big( k (f, x, z) - h (k (f, x, z)) \big)\, \n (f, dz),
        \end{split}
        \end{equation}
        for some \(\kappa > 0\) that satisfies \(h(y) = y\) for all \(\|y\| \leq \kappa\). Such an \(\kappa\) exists by definition of a truncation function. Further, the hypothesis \(\|k\| \leq \gamma\) yields \(\gamma > \kappa\) in case \(\|k\| > \kappa\). Hence, we get \eqref{eq: small gamma weg}, and consequently, the Lipschitz condition from Condition~\ref{cond: V3}~(iii), under (b). 
\end{remark}

We are in the position to formulate the main result of this section. It extends \cite[Lemma~2.38, Theorem~2.39]{CN22b} from a one-dimensional path-continuous setting to the present multidimensional setup that allows for jumps. 

\begin{theorem} \label{thm: main V}
    \quad
    \begin{enumerate}
        \item[\textup{(i)}] If Condition~\ref{cond: V1} holds, then \(v^*\) is a viscosity subsolution to the nonlinear PDE \eqref{eq: G eq}.
        \item[\textup{(ii)}] If Condition~\ref{cond: V1} holds, then \(v_*\) is a viscosity supersolution to the nonlinear PDE~\eqref{eq: G eq}.
        \item[\textup{(iii)}] If Condition \ref{cond: V3} holds, then \(v = v^* = v_*\) and \(v\) is the unique bounded viscosity solution to the nonlinear PDE \eqref{eq: G eq}.
    \end{enumerate}
\end{theorem}

Once (i) and (ii) are established, part (iii) of Theorem~\ref{thm: main V} follows from a comparison theorem for Hamilton--Jacobi--Bellman PDEs that was proved in \cite{hol16}, see Theorem~\ref{theo: comparison} for a formulation tailored to our setting. 
The following is a direct corollary to Theorem~\ref{thm: main V}.

\begin{corollary} \label{cor: c_b feller}
    Suppose that Condition \ref{cond: V3} holds. Then, for every \(\psi \in C_b (\bR^d; \bR)\), the map \((t, x) \mapsto S_t (\psi) (x)\) is continuous. In particular, the semigroup \((S_t)_{t \in \bR_+}\) has the \(C_b\)--Feller property. 
\end{corollary}

The uniform continuity property of the symbol entailed in Condition \ref{cond: V3} allows us to conclude the next result.
Let \( C_0 (\bR^d; \bR)\)  be the space of continuous functions \(\phi \colon \bR^d \to \bR\) that are vanishing at infinity in the sense that \(|\phi(x)| \to 0\) as \(\|x \| \to \infty\).

\begin{corollary} \label{cor: c_0 feller}
    Suppose that Condition \ref{cond: V3} holds. Then, the semigroup \((S_t)_{t \in \bR_+}\) has the \(C_0\)--Feller property, i.e., 
   \(S_t (C_0 (\bR^d;\bR) ) \subset C_0 (\bR^d; \bR)\) for any \(t \in \bR_+\).
\end{corollary}
\begin{proof}
	First of all, Theorem~\ref{thm: main V} yields that \(S_t (C_0(\bR^d; \bR)) \subset C_b (\bR^d; \bR)\). In turn, it suffices to prove that \(S_t (\psi)\) vanishes at infinity for every \(\psi \in C_0 (\bR^d;\bR)\) and \(t > 0\).
    We adapt the argument from \cite[Proposition 4.4]{schilling_symbol}.
    Let \(\psi \in C_0(\bR^d;\bR)\) and \(\varepsilon  > 0\).
    Since \(\psi\) is vanishing at infinity, there exists an \(r > 0\) such that \(|\psi(y)| \leq \varepsilon\) for \(\|y\| \geq r\).
    Let \(x \in \bR^d\) with \(\|x \| \geq 2 r\), and let \(P \in \cR(x)\). For \(t > 0\), we obtain 
    \begin{align*}
       E^P \big[ \big| \psi(X_t) \big| \big] & \leq \| \psi \|_\infty P \big( \|X_t\| > r\big) + \varepsilon.
    \end{align*}
    We proceed by bounding the first term on the right hand side. Notice that
    \begin{align*}
       P \big( \|X_t\| > r\big) &\leq P\big( \|X_t - x\| \geq \|x\|- r\big)
       \\& \leq P\Big( \sup_{s \in [0,t]} \|X_s - x\| \geq \|x\|/2\Big).
    \end{align*}
    Applying the maximal inequality (Proposition \ref{prop: maximal inequality} below) yields a constant \(C > 0\), independent of \(x \in \bR^d\), such that
    \begin{align*}
        P \big( \|X_t\| > r\big) & \leq C t\, \sup_{f \in F}  \sup_{\|y-x \| \leq \|x\|/2} \,\, \sup_{\|\xi\| \leq 2 / \|x\|} |\q(f,y,\xi)|
        \\& \leq C t \sup_{f \in F} \sup_{\|y\| \leq 3 \|x\|/2} \,\, \sup_{\|\xi\| \leq 2 / \|x\|} |\q(f,y,\xi)|.
    \end{align*}
    Taking the sup over \(P \in \cR(x)\) implies 
    \[
    | S_t(\psi)(x) | \leq C t \,\sup_{f \in F} \sup_{\|y\| \leq 3 \|x\|/2} \,\, \sup_{\|\xi\| \leq 2 / \|x\|} |\q(f,y,\xi)| + \varepsilon,
    \]
    for some constant \(C > 0\) that is independent of \(x \in \bR^d\).
    Hence, thanks to part (ii) of Condition \ref{cond: V3}, taking the limit \(\|x\| \to \infty\) yields
    \[
    \limsup_{\|x\| \to \infty} | S_t(\psi )(x) | \leq \varepsilon.
    \]
    As \(\varepsilon > 0\) was arbitrary, this completes the proof.
\end{proof}

In the following section, we proceed by investigating the relation of the operator \(G\) and the (pointwise) generator of \((S_t)_{t \in \bR_+}\).

\subsection{The Stochastic Representation Result}
For a sublinear Markovian semigroup \((T_t)_{t \in \bR_+}\) on a convex cone~\(\cH\) of functions from \(\bR^d\) into \(\bR\), its \emph{pointwise (infinitesimal) generator} \( A \colon  \cD(A) \to \cH \)  is defined by 
\begin{align*}
A(\phi)(x) & := \lim_{t \to 0} \frac{T_t(\phi)(x) - \phi(x)}{t}, \quad x \in \bR^d, \ \phi \in \cD(A), \\
\cD(A) &:= \Big\{ \phi \in \cH \colon \exists g \in \cH \text{ such that }  \lim_{t \to 0} \frac{T_t(\phi)(x) - \phi(x)}{t} = g(x) \ \ \forall x \in \bR^d \Big\}.
\end{align*}
Below, we will also use the following relaxation of the domain \(\cD(A)\). We define its \emph{pointwise domain} by 
\begin{align*}
\cD^p(A) := \Big\{ \phi \in \cH \colon \lim_{t \to 0} \frac{T_t(\phi)(x) - \phi(x)}{t} \ \text{ exists for all } x \in \bR^d \Big\}.
\end{align*}

For the sublinear semigroup \((S_t)_{t \in \bR_+}\) that is given by \eqref{eq: semigroup}, on a large class of test functions, the (pointwise) generator can be identified as the Hamilton--Jacobi--Bellman operator~\(G\) that is defined in \eqref{eq: G}.

\begin{proposition} \label{prop: generator}
Suppose that Condition \ref{cond: V1} holds.
Let \((S_t)_{t \in \bR_+}\) be the family of operators
defined in \eqref{eq: semigroup}. 
Then, for every \(\phi \in C_b^2(\bR^d;\bR)\) such that \(\nabla^2\phi \) is uniformly continuous,
we have 
\begin{align} \label{eq: G eq PG}
\lim_{t \to 0} \frac{S_t(\phi)(x) - \phi(x)}{t} = G(x, \phi), \quad x \in \bR^d.
\end{align}
\end{proposition}
\begin{proof}
It suffices to verify the prerequisites of \cite[Lemma 6.1]{K19}. This will imply the claim. By virtue of Condition \ref{cond: V1} and Proposition \ref{prop: existence}, we need to prove the following assertion: for every compact set \(K \subset \bR^d\), and every \(g \in C^1_b(\bR^d; \bR)\) with \(|g(x)| \leq \|x\|^2 \wedge 1\), there exists a modulus of continuity \(\ka = \ka (g, K) \colon [0, \infty] \to [0, \infty]\) such that 
    \[
    \Big| \int (g (k (f, x, z)) - g (k (f, y, z))) \, \n (f, dz) \Big| \leq \ka (\|x - y\|)
    \]
    for all \(f \in F\) and \(x, y \in  K\).
This follows from Lemma \ref{lem: non local 1} below.    
\end{proof}

\begin{remark}
    By Corollary \ref{cor: c_b feller}, the family \((S_t)_{t \in \bR_+}\)
    defines a sublinear Markovian semigroup on \(C_b(\bR^d;\bR)\).
    Further, notice that \eqref{eq: G eq PG} implies that \(\phi\) is in the pointwise domain of the pointwise generator of \((S_t)_{t \in \bR_+}\).
\end{remark}

As in the linear case, the (pointwise) generator \(A\) of a sublinear Markovian semigroup \((T_t)_{t \in \bR_+}\) can be linked to an evolution equation that is given by 
\begin{equation} \label{eq: PDE A}
\begin{cases}   
\partial_t u (t, x) - A(u(t, \cdot\, ))(x) = 0, & \text{for } (t, x) \in (0, \infty) \times \mathbb{R}^d, \\
u (0, x) = \psi (x), & \text{for } x \in \bR^d,
\end{cases}
\end{equation}
where \(\psi \colon \bR^d \to \mathbb{R}\) is a suitable function.

\smallskip

The following result makes this connection precise.
It is a slight extension of \cite[Proposition 4.10]{hol16}, where such a result was proved for the set \(\cD(A)\).
A careful inspection of the proof shows that the statement remains valid when \(\cD (A)\) is replaced by the larger set \(\cD^p(A)\).

\begin{proposition} \label{prop: weak sense}
Let \((T_t)_{t \in \bR_+}\) be a sublinear Markovian semigroup on a convex cone \(\cH\) of real-valued functions on \(\bR^d\)
containing all constant functions, and let \(A\)
be its pointwise generator, with pointwise domain \(\cD^p(A)\).
Consider \(\psi \in \cH\) and assume that \((t,x)\mapsto T_t(\psi)(x)\) is continuous.
If \(C^\infty_b(\bR^d; \bR) \subset \cD^p(A)\),
then 
\[
\bR_+ \times \bR^d \ni (t, x) \mapsto T_t(\psi)(x)
\]
is a viscosity solution to the evolution equation \eqref{eq: PDE A}.
\end{proposition}

\begin{remark}
Consider a sublinear Markovian semigroup on a convex cone \(\cH\) with generator \(A\).
By definition of the domain, the generator \(A\) maps \(\cD(A)\) into \(\cH\).
For the semigroup \((S_t)_{t \in \bR_+}\) on the space of \emph{bounded} continuous functions from \(\bR^d\) into \(\bR\), the limit 
\[
\lim_{t \to 0} \frac{S_t(\phi)(x) - \phi(x)}{t}
\]
equals \(G(x,\phi)\) for all \(\phi \in C^\infty_b(\bR^d;\bR)\) and \(x \in \bR^d\) by Proposition \ref{prop: generator}. Hence, the assumption \(C^\infty_b(\bR^d; \bR) \subset \cD(A)\) from \cite[Proposition 4.10]{hol16} entails in particular that \(x \mapsto G(x,\phi)\) is bounded for every \(\phi \in C^\infty_b(\bR^d; \bR)\).
This imposes boundedness restrictions on the coefficients \((b,\sigma \sigma^*, \n \o k^{-1})\).
Using the pointwise domain \(\cD^p(A)\) in Proposition \ref{prop: weak sense} allows us to circumvent them. 
\end{remark}

By virtue of Theorem~\ref{thm: main V}, Proposition~\ref{prop: weak sense} enables us to identify \((S_t)_{t \in \bR_+}\) as the unique semigroup whose (pointwise) generator is given through the operator \(G\). More precisely, we have the following:

\begin{theorem} \label{thm: uni chara}
Suppose that Condition \ref{cond: V3} holds.
Then,
\((S_t)_{t \in \bR_+}\) as defined in~\eqref{eq: semigroup} is the unique sublinear Markovian semigroup on \(C_b(\bR^d; \bR)\) with the following properties:
\begin{enumerate}
    \item[\textup{(i)}] for \(\psi \in C_b(\bR^d; \bR)\), the map \((t,x) \mapsto S_t(\psi)(x)\) is continuous;
    \item[\textup{(ii)}] for all \(x \in \bR^d\) and \(\phi \in  C_c^\infty(\bR^d; \bR)\),
    \[
    \lim_{t \to 0} \frac{S_t (\phi)(x) - \phi (x)}{t} = G(x, \phi).
    \]
\end{enumerate}
\end{theorem}
\begin{proof}
    That \((S_t)_{t \in \bR_+}\) given by \eqref{eq: semigroup} has the properties (i) and (ii) follows from Theorem~\ref{thm: main V} and Proposition~\ref{prop: generator}.

    Conversely, suppose that \((T_t)_{t \in \bR_+}\) is a sublinear Markovian semigroup with the properties (i) and (ii). 
    We claim that 
    \begin{equation} \label{eq: pf cor}
     \lim_{t \to 0} \frac{T_t (\phi)(x) - \phi (x)}{t} = G(x, \phi), \quad x \in \bR^d,  
    \end{equation}
    extends to all \(\phi \in C^\infty_b(\bR^d;\bR)\). To see this, observe that, for fixed \(x \in \bR^d\), we have the tightness property
    \[
    \lim_{R \to \infty} \, \sup_{f \in F} \big\{ \n\, (f, \{ \| k (f, x, \, \cdot\, )\| > R \} ) \big\} = 0,
    \]
    as a consequence of Condition \ref{cond: V3}. Hence, it follows from the proof of \cite[Corollary 4.7]{K21} that \eqref{eq: pf cor} remains valid for \(\phi \in C^\infty_b(\bR^d;\bR)\). In particular, we conclude that 
    \(C^\infty_b(\bR^d;\bR)\) is contained in the pointwise domain of the      generator of \((T_t)_{t \in \bR_+}\).
    Next, take an arbitrary \(\psi \in C_b(\bR^d; \bR)\).
    Thanks to part (i) and Proposition \ref{prop: weak sense}, we conclude that the map \((t, x) \mapsto T_t(\psi)(x)\)
    is a (bounded) viscosity solution to \eqref{eq: G eq}.
    Therefore, part (iii) of Theorem~\ref{thm: main V} implies that
    \[
    S_t(\psi)(x) = T_t(\psi)(x), \quad (t,x) \in \bR_+ \times \bR^d,
    \]
    as desired.  
\end{proof}

To the best of our knowledge, the first theorem that uses viscosity methods to characterize sublinear semigroups via versions of the pointwise generator appeared in \cite{lions_nisio}. A general characterization of sublinear convolution semigroups (that correspond to nonlinear L\'evy processes) was proved in \cite{K21}. A related result for nonlinear Markov processes with jumps was recently established in our paper \cite{CN23a}, using a different methodology as in the present paper. Compared to this result (\cite[Corollary~3.14]{CN23a}), Theorem~\ref{thm: uni chara} requires less assumptions on the coefficients, i.e., the action space \(F\) needs not to be compact, the coefficients might be unbounded, less integrability conditions are imposed and the correspondence \(\Theta\) is not assumed to be convex-valued. The latter point is particularly useful from a modeling point of view, as it allows to treat  \(b, \sigma, k\) and \(\{\n \, (f, dz) \colon f \in F\}\) directly, without taking the convexity of \(\Theta\) into consideration.
At this point, let us mention that the paper \cite{CN23a} works with a single reference measure \(\n\), instead of a family \(\{\n \, (f, dz) \colon f \in F\}\). From a modeling point of view this is no restriction, as it is always possible to use a single reference measure (cf., e.g., \cite[Remark 3.6, p. 194]{CJ81} or \cite[Exercise~9.2.6]{stroockanalytic}). However, to verify the Lipschitz assumptions for the coefficient \(k\), it appears to be handy in certain cases to allow the reference measure to depend on \(F\). For example, we think of situations of the type
\[
F \subset \bR_+, \qquad \int \1_{G \backslash \{0\}} ( k (x, z) ) \, f \cdot\, \n (dz), 
\]
where, in our setting, assumptions can be imposed directly on \(k\), while computations are needed to translate such cases into the framework from \cite{CN23a}.

\section{Proofs} \label{sec: proofs}

This section is dedicated to the proofs of our results.
We collect several auxiliary results in Section \ref{sec: maximal inequality}. We then proceed by establishing continuity properties of the nonlinear operator \(G\) in Section~\ref{sec: pf continuity G}.
Thereafter, in Sections \ref{sec: pf subsolution} and \ref{sec: pf supersolution}, we verify the sub- and supersolution properties of the upper and lower envelopes of \(v\).
Finally, leaning on the results obtained in previous sections, we prove Theorem \ref{thm: main V} in Section \ref{sec: pf thm viscosity}.

\subsection{Maximal Inequality} \label{sec: maximal inequality}
The following result is a consequence of \cite[Proposition 5.1]{K19}.
It will turn out to be of great value for the proof of our main result Theorem \ref{thm: uni chara}, as it allows us to relax integrability assumptions on the coefficients compared to \cite[Corollary~3.14]{CN23a}.

\begin{proposition} \label{prop: maximal inequality}
Suppose that Condition \ref{cond: V1} holds.
There exists a constant \(C > 0\) such that
\[
\sup_{P \in \cR(x)} P \Big( \sup_{s \in [0,u]} \|X_s - x \| > r \Big) \leq u\, C  \sup_{f \in F} \sup_{\|y-x \| \leq r} \sup_{\|\xi\| \leq r^{-1}} |\q(f,y,\xi)|
\]
for all \(u \in \bR_+\) and \(x \in \bR^d\).
\end{proposition}

We now record two consequences of Proposition \ref{prop: maximal inequality} that will be useful in the following.
\begin{corollary} \label{cor: maximal inequality}
Suppose that Condition \ref{cond: V1} holds, and let \(K \subset \bR^d\) be compact. Then,
\begin{itemize}
    \item[\textup{(i)}] for every \(r > 0\),
\[
\sup_{x \in K} \sup_{P \in \cR(x)} P \Big( \sup_{s \in [0,u]} \|X_s - x \| > r\Big) \to 0 \text{ as } u \to 0,
\]
    \item[\textup{(ii)}] and, for every \(u > 0\),
    \[
\sup_{x \in K} \sup_{P \in \cR(x)} P \Big( \sup_{s \in [0,u]} \|X_s - x \| > r\Big) \leq u\, C_r
\]
for some constant \(C_r >0\) such that \(C_r \to 0\) as \(r \to \infty\).
\end{itemize}
\end{corollary}
\begin{proof}
We start with part (i).
By parts (i) and (iii) of Condition \ref{cond: V1}, we have
\[
\sup_{x \in K} \sup_{f \in F} \sup_{\|y-x \| \leq r} \sup_{\|\xi\| \leq r^{-1}} |\q(f,y,\xi)| < \infty
\] 
for every \(r > 0\) and every compact subset \(K \subset \bR^d\). Hence, the assertion of~(ii) follows from Proposition~\ref{prop: maximal inequality}.

\smallskip 
Regarding part (ii), using Proposition \ref{prop: maximal inequality}, it suffices to show that, for fixed \(u >0\),
\[
\sup_{x \in K} \sup_{f \in F} \sup_{\|y-x \| \leq r} \sup_{\|\xi\| \leq r^{-1}} |\q(f,y,\xi)| \to 0 \text{ as } r\to \infty.
\] 
This is again a consequence of the uniform continuity of \(\q\) (part (ii) of Condition \ref{cond: V1}). 
The proof is complete.
\end{proof}

The next result is a restatement of \cite[Theorem 5.3 (ii)]{K19}.
\begin{proposition} \label{prop: continuity in time}
Suppose that Condition \ref{cond: V1} holds, and let \(\phi \in C_b(\bR^d; \bR)\).
Then, the map \(t \mapsto S_t(\phi)(x)\) is continuous, uniformly in \(x \in K\) for any compact subset \(K \subset \bR^d\).
\end{proposition}

\begin{lemma} \label{lem: initial values}
Suppose that Condition \ref{cond: V1} holds. Then,
\[
v^*(0,x) = \psi(x) \, \, \text{ and } \, \, v_*(0,x) = \psi(x), \quad x \in \bR^d.
\]
\end{lemma}
\begin{proof}
We only detail the proof for the upper envelope \(v^*\). The lower envelope can be handled with a similar argument.
For \(x \in \bR^d\), take a sequence \((t^n,x^n)_{n \in \bN} \subset \bR_+ \times \bR^d\) such that 
\((t^n, x^n) \to (0,x)\) and \(\lim_{n \to \infty} v(t^n,x^n) = v^*(0,x)\).
Recall that, by construction, \(v(0,y) = \psi(y)\) for every \(y \in \bR^d\). Hence,
\begin{align*}
    | v^*(0,x) - \psi(x) | &= \lim_{n \to \infty} | v(t^n,x^n) - \psi(x) | 
    \\&\leq  \lim_{n \to \infty} \big( | v(t^n,x^n) - v(0,x^n) | + | \psi(x^n) - \psi(x) | \big). 
\end{align*}
Since \(\psi\) is continuous, the second term vanishes. Further, using Proposition~\ref{prop: continuity in time}, we observe that the first term vanishes, too.
This completes the proof.
\end{proof}

\subsection{Continuity of \(G\)} \label{sec: pf continuity G}

In this section, we will show that for \(\phi \in C^{2}_b(\bR_+ \times \bR^d; \bR)\) the map \((t,x) \mapsto G(x, \phi(t,\cdot\,))\) is continuous, where the operator \(G\) is defined in \eqref{eq: G}.
We start with two estimates for the non-local part of \(G\) that will also be used in the proof of Theorem~\ref{thm: main V}.

\begin{lemma} \label{lem: non local 1}
    Suppose that Condition \ref{cond: V1} holds.
    Let \(g \colon \bR_+ \times \bR^d \times \bR^d \to \bR\) be continuous and bounded. Assume that there exists a constant \(C > 0\) such that
    \begin{itemize}
        \item[\textup{(i)}] \(|g(t,x,z)| \leq C ( \|z\|^2 \wedge 1)\),

        \item[\textup{(ii)}] and 
        \[
        \big| g(t,x,z) - g(t,x,w) \big| \leq C \|z-w\|,
        \]
        for all \((t,x,z,w) \in \bR_+ \times \bR^d \times \bR^d \times \bR^d\).
\end{itemize}
        Then, for every compact set \(K \subset \bR^d\) and every \(\varepsilon > 0\), there exists a constant \(\delta = \delta(K, \varepsilon) > 0\) 
        such that
         \[
         \int \big| g(t^0, x^0, k (f, x, z)) - g (t^0, x^0, k(f, y, z)) \big| \, \n (f, dz)  \leq \varepsilon
        \]
         for all \(f \in F\), \((t^0, x^0) \in  \bR_+ \times \bR^d\) and \(x,y \in K\) with \(\| x-y\| \leq \delta.\)
\end{lemma}
\begin{proof}
Let \(K\subset \bR^d\) be compact, \((t^0, x^0) \in \bR_+ \times \bR^d\) and \(x,y \in K\). Further, let \(\kappa  \in (0,1)\) and \(R > 1\).   
We are going to separately estimate the following four quantities
\begin{align*}
    &I_1 :=  \int_{\{\gamma \leq \kappa\}} \big| g(t^0,x^0, k (f, x, z)) - g (t^0,x^0,k (f, y, z)) \big| \, \n (f, dz), 
    \\&  I_2 :=  \int_{\{\kappa < \gamma \leq R\}} \big| g (t^0,x^0, k (f, x, z)) - g (t^0,x^0, k (f, y, z)) \big| \, \n (f, dz),
    \\&  I_3:=  \int_{\{\gamma > R\}} \big| g (t^0,x^0,k (f, x, z)) - g (t^0,x^0,k (f, y, z)) \big| \, \n (f, dz),
\end{align*}
where \(\gamma = \gamma \, (K) \colon L \to [0,\infty]\) is as in Condition \ref{cond: V1}. Let \(\ka = \ka \, (K) \colon [0, \infty] \to [0, \infty]\) be the modulus of continuity of the coefficient \(k\) from Condition \ref{cond: V1}. 
We start with \(I_1\). Using the growth assumption on \(g\), we obtain 
\begin{equation} \label{eq: estimate i_1}
    \begin{split}
        I_1 & \leq \int_{\{\gamma \leq \kappa\}} \big(\|k(f,x,z)\|^2 \wedge 1 + \|k(f,y,z)\|^2 \wedge 1\big) \, \n(f, dz)
        \\& \leq C \int_{\{\gamma \leq \kappa\}} (\gamma^2(z) \wedge 1) \, \n(f, dz),
    \end{split}
\end{equation}
where the constant \(C > 0\) depends on \(K\). In particular, by the tightness assumption
\[
\lim_{\kappa \searrow 0}\, \sup_{f \in F} \, \int_{\{\gamma \leq \kappa\}} \gamma^2 (z) \, \n (f, dz) = 0
\]
from Condition~\ref{cond: V1}, we get that \(I_1 \to 0\) as \(\kappa \searrow 0\). We proceed with \(I_2\). Using the Lipschitz continuity of \(g\) and Condition~\ref{cond: V1}, we obtain that 
\begin{equation} \label{eq: estimate i_2}
    \begin{split}
        I_2 & \leq C \int_{\{\kappa < \gamma \leq R\}} \|k(f,x,z) - k(f,y,z)\| \, \n(f, dz)
        \\& \leq C \ka(\|x-y\|) \int_{\{\kappa < \gamma \leq R\}} \gamma(z)  \, \n(f, dz),
        \\& \leq C \ka(\|x-y\|) \, \Big(\dfrac{1}{\kappa} + R \Big)\, \int \,  (\gamma^2(z) \wedge 1)  \, \n(f, dz),
    \end{split}
\end{equation}
where the constant \(C > 0\) depends on \(g\).
Finally, using the boundedness of \(g\), it follows that
\begin{equation} \label{eq: estimate i_4}
    \begin{split}
        I_3 & \leq C \, \n (f, \{\gamma > R\}),
    \end{split}
\end{equation}
where the constant \(C > 0\) depends on \(g\). In particular, by the tightness assumption 
\[
\lim_{R \to \infty} \, \sup_{f \in F} \, \n \, (f, \{\gamma > R\}) = 0
\]
from Condition~\ref{cond: V1}, we get \(I_3 \to 0\) as \(R \to \infty\).
Combining \eqref{eq: estimate i_1} -- \eqref{eq: estimate i_4}, it follows that for every \(\varepsilon > 0\) there exists a constant \(\delta > 0\) such that, for all \(f \in F\) and all \(x,y \in K\) with \(\|x-y\| \leq \delta\), the estimate 
\[
 \Big| \int (g (t^0,x^0,k (f, x, z)) - g (t^0,x^0,k (f, y, z))) \, \n (f, dz) \Big| \leq \varepsilon
\]
holds. This completes the proof.
\end{proof}

For \(\phi \in  C^{1,2}_b(\bR_+ \times \mathbb{R}^d; \bR) \), we define the function
\(\g = \g(\phi) \colon \bR_+ \times \bR^d \times \bR^d \to \bR\) by
		\begin{equation} \label{eq: def g}
			\g(t,x,z) := \phi(t,x+z) - \phi(t,x) - \langle \nabla \phi(t,x), h(z) \rangle.    
		\end{equation}
		For notational convenience, we usually suppress the dependence of \(\g\) on \(\phi\) in the following.

\begin{lemma} \label{lem: non local 2}
    Suppose that Condition \ref{cond: V1} holds and let \( \phi \in  C^{2}_b(\bR_+ \times \mathbb{R}^d; \bR)\).
     Then, for every compact subset \(K \subset \bR^d\) and every \(\varepsilon > 0\), there exists a constant \(\delta = \delta(K, \varepsilon) > 0\) 
        such that
         \[
         \int \big| \g(t, x, k (f, x^0, z)) - \g (s, y, k(f, x^0, z)) \big| \, \n (f, dz)  \leq \varepsilon
        \]
         for all \(f \in F\), \(x^0 \in K\), \(t,s \in \bR_+\) and \(x,y \in \bR^d\) with 
         \(
         |t-s| + \| x-y\| \leq \delta.
         \)
\end{lemma}
\begin{proof}
 As \(\phi \in C^2_b (\bR_+ \times \bR^d; \bR)\), there exists a modulus of continuity \(\ka \colon [0, \infty] \to [0, \infty]\) such that 
\[
| \phi (t, x) - \phi (s, y) | + \| \nabla \phi (t, x) - \nabla \phi (s, y) \| \leq \ka (|t - s| + \|x - y\|)
\]
for all \((t, x), (s, y) \in \bR_+ \times \bR^d\).
From now on, fix an \(\varepsilon > 0\).

Thanks to the tightness assumptions from Condition~\ref{cond: V1}, there exist constants \(\kappa = \kappa (\varepsilon) \in (0, 1)\) and \(R = R (\varepsilon) > 1\) such that 
\[
\sup_{f \in F} \, \int_{\{\gamma \leq \kappa\}} \gamma^2 (z) \, \n \, (f, dz) \leq \varepsilon, \quad \sup_{f \in F} \, \n \, (f, \{\gamma > R\}) \leq \varepsilon.
\]
W.l.o.g., we can take \(\kappa \in (0, 1)\) such that \(h (x) = x\) for \(\|x\| \leq \kappa\).
Fix \(f \in F\), a compact set \(K \subset \bR^d\), \(x^0 \in K\), \(t,s \in \bR_+\) and \(x,y \in \bR^d\), 
and set 
\begin{align*}
I_1 &:= \int_{\{\gamma  \leq \kappa\}} \big| \g(t, x, k (f, x^0, z)) - \g (s, y, k(f, x^0, z)) \big| \, \n (f, dz),
\\
I_2 &:= \int_{\{\kappa < \gamma \leq R\}} \big| \g(t, x, k (f, x^0, z)) - \g (s, y, k(f, x^0, z)) \big| \, \n (f, dz), 
\\
I_3 &:= \int_{\{\gamma  > R\}} \big| \g(t, x, k (f, x^0, z)) - \g (s, y, k(f, x^0, z)) \big| \, \n (f, dz),
\end{align*}
where \(\gamma = \gamma \, (K) \colon L \to \bR_+\) is as in Condition~\ref{cond: V1}.
In the following, we estimate each of these terms separately. 
We deduce from Taylor's theorem that  
\begin{align*}
 \Big| \big(\phi(t&, x + k(f,x^0,z))  - \phi(t, x) - \langle \nabla \phi(t, x),   k(f,x^0,z) \rangle \big)  
\\& \qquad  -  \big(\phi(s, y +  k(f,x^0,z))  - \phi(s, y) - \langle \nabla \phi(s, y),   k(f,x^0,z) \rangle \big) \Big|
\leq C \, \| k(f,x^0,z)\|^2, \phantom \int
\end{align*}
for every \(z \in L\), where the constant only depends on \(\phi\).
Hence, by our choice of \(\kappa\), we obtain that
\begin{align} \label{eq: I1 - v2}
	I_1 &\leq C\, \, \int_{\{\gamma \leq \kappa\}} \gamma^2 (z) \, \n (f, dz) \leq C \, \varepsilon.
\end{align}

Next, we estimate \(I_2\).
For all \(u \in \mathbb{R}^d\), we have
\begin{align*}
 \big| \langle \nabla \phi (t, x) - \nabla \phi (s, y), h(u) \rangle \big| &\leq C \, \ka  (|t - s| + \|x - y\|),
\end{align*}
where the constant \(C > 0\) depends on \(h\).
 Moreover, for all \(z \in L\), 
\begin{align*} 
\Big|  \big(\phi(t, x + k(f,x^0,z)) &- \phi(t, x) \big) -  \big(\phi(s, y + k(f,x^0,z))  - \phi(s, y)\big) \Big| \leq 2 \, \ka (|t - s| + \|x - y\|).
\end{align*}
Consequently, we get that 
\begin{equation} \begin{split} \label{eq: I2 - v2}
	I_2 &\leq \ka (|t - s| + \|x - y\|) \, C \, \n (f, \{\kappa < \gamma \leq R\} ) 
 \\&\leq \ka (|t - s| + \|x - y\|) \, C \, \int ( \gamma^2 (z) \wedge 1 )\, \n (f, dz),
\end{split}
\end{equation}
where the constant might depend on \(\kappa\) but it is independent of \(R\). 

Finally, we estimate \(I_3\). Notice that \(\g\) is bounded. Hence, we obtain 
\begin{align} \label{eq: l3 - v3}
	I_3 \leq C \, \n (f, \{\gamma > R\}) \leq C \, \varepsilon, 
	\end{align}
by the choice of \(R\). Here, the constant \(C > 0\) only depends on \(\g\), i.e., on \(\phi\) and \(h\).
All together, \eqref{eq: I1 - v2} -- \eqref{eq: l3 - v3} yield the existence of a \(\delta > 0\) such that, for all \((t, x), (s, y) \in \bR_+ \times \bR^d\) with \(|t - s| + \|x - y\| \leq \delta\), 
\[
 \int \big| \g(t, x, k (f, x^0, z)) - \g (s, y, k(f, x^0, z)) \big| \, \n (f, dz) \leq C \, \varepsilon,
\]
where the constant \(C > 0\) only depends on \(\phi\) and \(h\).
This implies the claim. 
\end{proof}

\begin{lemma} \label{lem: G cont}
    Suppose that Condition \ref{cond: V1} holds.
    For every \(\phi \in C^{2}_b(\bR_+ \times \mathbb{R}^d; \bR)\), the map
    \((t,x) \mapsto G(x, \phi(t,\cdot\,))\)
    is continuous.
\end{lemma}
\begin{proof}
Fix \(\phi \in  C^{2}_b(\bR_+ \times \mathbb{R}^d; \bR)\).
Notice that, for \(s,t \in \bR_+\) and \(x,y \in \bR^d\),
\begin{align*}
    \big| G(x,\phi(t,\cdot\,)) &- G(y, \phi(s,\cdot\,)) \big|  
    \\&\leq \sup \Big\{ 
    \big| \langle \nabla \phi (t,x), b (f, x) \rangle - \langle \nabla \phi (s,y), b (f, y) \rangle \big|
    \\& \hspace{2cm} + \tfrac{1}{2} \big| \on{tr} \big[ \nabla^2\phi (t, x) \sigma \sigma^* (f, x) \big] - \on{tr} \big[ \nabla^2\phi (s, y) \sigma \sigma^* (f, y) \big] \big| \phantom \int
    \\& \hspace{2cm} + \int \big| \g(t,x,k(f,x,z)) - \g(s,y,k(f,y,z)) \big| \, \n(f, dz)
    \colon f \in F \Big\}.
\end{align*}
The local term 
\begin{align*}
    \big| \langle \nabla \phi (t,x), & b (f, x) \rangle - \langle \nabla \phi (s,y), b (f, y) \rangle \big|
     \\&+ \tfrac{1}{2} \big| \on{tr} \big[ \nabla^2\phi (t, x) \sigma \sigma^* (f, x) \big] - \on{tr} \big[ \nabla^2\phi (s, y) \sigma \sigma^* (f, y) \big] \big|
\end{align*}
vanishes, uniformly in \(f \in F\), as \(t \to s\) and \(x \to y\), by part (i) and (iii) of Condition \ref{cond: V1}. Further, using Lemmata \ref{lem: non local 1} and \ref{lem: non local 2}, we see that the non-local part, i.e., 
\[
\sup_{f \in F} \int \big| \g(t,x,k(f,x,z)) - \g(s,y,k(f,y,z)) \big| \, \n(f, dz),
\]
vanishes, too.
This completes the proof.
\end{proof}

\subsection{Subsolution Property} \label{sec: pf subsolution}
In this section we establish the subsolution property of \(v^*\), i.e., part (i) of Theorem~\ref{thm: main V}. The argument borrows ideas from the proofs of \cite[Lemma~2.38]{CN22b} and \cite[Lemma~6.1]{K19}. 
\begin{lemma} \label{lem: en sub}
Suppose that Condition~\ref{cond: V1} holds. Then, \(v^*\) is a viscosity subsolution to the nonlinear PDE \eqref{eq: G eq}.
\end{lemma} 
\begin{proof}
First, observe that \(v^*(0,x) = \psi(x)\) by Lemma \ref{lem: initial values}.
Let \(\phi \in  C^\infty_b(\bR_+ \times \bR^d; \bR)\) such that \(\phi \geq v^*\) and \(\phi (t^0, x^0) = v^*(t^0, x^0)\) for some \((t^0, x^0) \in (0, \infty) \times \bR^d\). 
There exists a sequence \((t^n, x^n)_{n = 1}^\infty \subset \bR_+ \times \bR^d\) such that \((t^n, x^n) \to (t^0, x^0)\) and 
\[
v^* (t^0, x^0) = \lim_{n \to \infty} v (t^n, x^n). 
\]
We take an arbitrary \(0 < u < t_* := \inf_{n \in \mathbb{Z}_+} t^n\), where, w.l.o.g., we can assume that \(t_* > 0\).
By the semigroup property of \((S_t)_{t \in \bR_+}\), for every \(n \in \mathbb{N}\), we obtain that 
\begin{equation} \label{eq: main ev upper}
	\begin{split}
		0 &= \sup_{P \in \cR (x^n)} E^P \big[ v (t^n - u, X_u) \big] - v (t^n, x^n)
		\\&\leq \sup_{P \in \cR (x^n)} E^P \big[ v^* (t^n - u, X_u) \big] - v^* (t^0, x^0) + v^* (t^0, x^0) - v (t^n, x^n)
  \\&= \sup_{P \in \cR (x^n)} E^P \big[ v^* (t^n - u, X_u) \big] - \phi (t^0, x^0) + v^* (t^0, x^0) - v (t^n, x^n)
		\\&\leq \sup_{P \in \cR (x^n)} E^P \big[ \phi (t^n - u, X_u) \big] - \phi (t^n, x^n)  \\&\hspace{3cm} + \phi (t^n, x^n) - \phi (t^0, x^0) + v^* (t^0, x^0) - v (t^n, x^n).
	\end{split}
\end{equation}
Fix \(n \in \bN\), and define, for \(r > 0\), the stopping time 
\begin{align*}
\tau^n_r := \inf \Big\{t \geq 0 \colon \sup_{u \in [0, t]} \|X_u - x^n \| \geq r \Big\}.
\end{align*}
Take \(P \in \cR(x^n)\) and denote the Lebesgue densities of the \(P\)-characteristics of \(X\) by \((b^P, a^P, \nu^P)\).
It\^o's formula yields that \(P\)-a.s.
\begin{align*}
\phi (t^n - (u \wedge \tau^n_r), X_{u \wedge \tau^n_r}) & - \phi (t^n, x^n)  = - \int_0^{u\wedge \tau^n_r} \partial_t \phi (t^n - s, X_s) ds
\\&+ \int_0^{u\wedge \tau^n_r} \Big[ \langle b^P_s,  \nabla \phi (t^n - s, X_s) \rangle + \tfrac{1}{2} \on{tr} \big[ \nabla^2 \phi (t^n - s, X_s) a^P_s \big]   \Big] ds
\\&+ \int_0^{u\wedge \tau^n_r} \int  \g(t^n - s, X_{s-}, z) \mu^X(ds,dz)
\\&+ \text{local \(P\)-martingale}, \phantom \int
\end{align*}
where the function \(\g\) is defined in \eqref{eq: def g}.
Thanks to the local boundedness Condition \ref{cond: V1} and the definition of \(\tau^n_r\), 
it follows that the local \(P\)-martingale part is actually a true \(P\)-martingale. Hence, for all \(n \in \mathbb{N}\), we obtain that 
\begin{align*}
E^P \big[ & \phi (t^n - (u \wedge \tau^n_r), X_{u \wedge \tau^n_r}) - \phi (t^n, x^n) \big] 
\\&= E^P \Big[ - \int_0^{u\wedge \tau^n_r}  \partial_t \phi (t^n - s, X_s) + \langle b^P_s, \nabla \phi (t^n - s, X_s) \rangle + \tfrac{1}{2} \on{tr} \big[ \nabla^2 \phi (t^n - s, X_s) a^P_s \big] ds \Big]
\\&\qquad \quad + E^P \Big[ \int_0^{u\wedge \tau^n_r} \int \g(t^n - s, X_{s}, z) \nu^P(ds,dz) \Big]
\\&=: I^n_u. \phantom \int
\end{align*}
For \(n \in \bN\), set
\begin{equation} \label{eq: def j^n}
\begin{split}
    J^n_1(u,P) & := E^P \Big[ \int_0^{u \wedge \tau^n_r} \big| \partial_t \phi (t^n - s , X_s) - \partial_t \phi (t^n, x^n) \big| ds\Big], \\
    J^n_2(u,P) & := E^P \Big[ \int_0^{u \wedge \tau^n_r} \big| \langle b^P_s, \nabla \phi (t^n - s , X_s) - \nabla \phi (t^n, x^n) \rangle \big| ds \Big], \\
    J^n_3(u,P) & := E^P \Big[ \int_0^{u \wedge \tau^n_r}  \big| \on{tr} \big[ ( \nabla^2 \phi (t^n - s , X_s) - \nabla^2 \phi (t^n, x^n) ) a^P_s \big ] \big| ds \Big], \\
    J^n_4(u,P) & :=  E^P \Big[ \int_0^{u \wedge \tau^n_r} \int \big| 
    \g(t^n-s, X_{s}, z) - \g(t^n, x^n, z) \big|
    \nu^P(ds,dz) \Big].
\end{split}
\end{equation}
Notice that
\begin{align*}
     - E^P \Big[\int_0^{u \wedge \tau^n_r}  \partial_t \phi (t^n - s, X_s) ds \Big] &\leq - u \partial_t \phi (t^n, x^n)  + J^n_1(u,P),\\
     E^P \Big[\int_0^{u \wedge \tau^n_r} \langle b^P_s, \nabla \phi (t^n - s, X_s) \rangle ds \Big] &\leq  E^P \Big[ \int_0^{u \wedge \tau^n_r} \langle b^P_s, \nabla \phi (t^n, x^n) \rangle  ds \Big] + J^n_2(u,P), \\
     E^P \Big[\int_0^{u \wedge \tau^n_r} \on{tr} \big[ \nabla^2 \phi (t^n - s, X_s) a^P_s \big] ds \Big] &\leq E^P \Big[ \int_0^{u \wedge \tau^n_r} \on{tr} \big[\nabla^2 \phi (t^n, x^n) a^P_s \big] ds \Big] + J^n_3(u,P), 
     \end{align*}
     and 
     \begin{align*}
    E^P \Big[ \int_0^{u \wedge \tau^n_r} \int \g(t^n - s, X_{s}&, z) \nu^P(ds,dz) \Big] 
    \\&\leq 
    E^P \Big[ \int_0^{u \wedge \tau^n_r} \int  \g(t^n, x^n, z)  \nu^P(ds,dz) \Big] + J^n_4(u,P).
\end{align*}
Hence, we get that
\begin{equation} \label{eq: bound i^n}
	\begin{split}
 I^n_u \leq - u \partial_t \phi(t^n, x^n) + E^P \Big[ \int_0^{u \wedge \tau^n_r} G^n (X_s) ds \Big] &+ J^n_1(u,P) + J^n_2(u,P) 
 \\&+ \tfrac{1}{2} J^n_3(u,P) + J^n_4(u,P)
 \end{split}
\end{equation}
for all \(n \in \mathbb{N}\),
where 
\begin{equation} \begin{split}
  G^n (x) := \sup \Big\{ \langle b (f, x),  \nabla  \phi (t^n, x^n) \rangle &+ \tfrac{1}{2} \on{tr} \big[ \nabla^2 \phi (t^n, x^n) \sigma \sigma^*(f,x) \big ] 
  \\&+ \int\, \g (t^n, x^n, k (f, x, z) ) \, \n (f, dz) \colon f \in F \Big\}.
\end{split} \end{equation}
As \(\phi \in C^\infty_b(\bR_+ \times \bR^d; \bR)\), the derivatives \(\partial_t \phi, \nabla \phi\) and \(\nabla^2 \phi\) are (globally) Lipschitz continuous jointly in time and space. In particular, they are uniformly continuous.
Hence, we can choose for every \(\varepsilon>0\) a constant \(\delta > 0\) such that
\begin{equation} \label{eq: eps delta}
    \begin{split}
\| \nabla^k \phi(t,x) - \nabla^k \phi(s,y) \| &+ | \partial_t \phi(t,x) - \partial_t \phi(s,y) |
 \leq \varepsilon
    \end{split}
\end{equation}
for \(k = 0, 1, 2\) and all \((t, x), (s, y) \in \bR_+ \times \bR^d\) such that
\(
|s-t| + \| x - y \| \leq \delta.
\)
Hence, we obtain, for \(0 < u < \tfrac{\delta}{2} \wedge t_*\), that 
\begin{equation} \label{eq: bound j_1}
\begin{split}
J^n_1(u,P) &\leq \int_0^u \Big( \varepsilon + 2 \| \partial_t \phi \|_\infty P \Big( \|X_s - x^n\| > \delta/2 \Big) \Big) ds 
\\& \leq C_1 u\Big( \varepsilon + \sup_{Q \in \cR(x^n)} Q\Big( \sup_{s \in [0,u]} \|X_s - x^n\| > \delta/2\Big)\Big),
\end{split}  
\end{equation}
for some constant \(C_1 > 0\) that is independent of \(n \in \bN\).
Similarly, using also the local boundedness Condition~\ref{cond: V1}, we obtain that 
\begin{equation} \label{eq: bound j_2}
    \begin{split}
        J^n_2(u,P) &\leq \Big(\sup_{f \in F} \sup_{\|y-x^n\| \leq r} \|b(f,y)\| \Big) \Big(\varepsilon u +  2 u \|\nabla \phi \|_\infty P\Big( \sup_{s \in [0,u]} \|X_s - x^n\| > \delta/2 \Big)\Big) 
 \\& \leq C_2 u\Big( \varepsilon + \sup_{Q \in \cR(x^n)} Q\Big( \sup_{s \in [0,u]} \|X_s - x^n\| > \delta/2 \Big)\Big),
    \end{split}
\end{equation}
and that 
\begin{equation} \label{eq: bound j_3}
    \begin{split}
        J^n_3(u,P) \leq C_3 u \Big( \varepsilon + \sup_{Q \in \cR(x^n)} Q\Big( \sup_{s \in [0,u]} \|X_s - x^n\| > \delta/2\Big)\Big),
    \end{split}
\end{equation}
where, again, \(0 < u < \tfrac{\delta}{2} \wedge t_*\) and the constants \(C_2, C_3 >0\) are independent of \(n \in \bN\) due to the boundedness of \(\{x^n \colon n \in \bN\}\).

Regarding the estimate of \(J^n_4(u,P)\), it follows as in Step 1 from the proof of \cite[Lemma~6.1]{K19}
that there exists a constant \(C_4 > 0\) such that
\[
\big| \g(t^n - s, X_{s}, z) - \g(t^n, x^n, z) \big| \leq C_4 (\|z\|^2 \wedge 1) \big( \varepsilon  + \1_{\{\| X_{s} - x^n\| \, >\, \delta/2\}} \big)
\]
for all \(s < u < \frac{\delta}{2} \wedge t_*\).
Hence, we obtain the estimate
\begin{equation} \label{eq: bound j_4}
    \begin{split}
            J^n_4 &\leq  C_4  u \, \Big(\sup_{f \in F} \sup_{\|y-x^n\| \leq r} \int ( \| k(f,y,z) \|^2 \wedge 1) \, \n(f, dz) \Big)  
 \\&\hspace{5cm}   \Big( \varepsilon   + P\Big( \sup_{s \in [0,u]} \|X_s - x^n \| > \delta/2 \Big) \Big)
    \\& \leq C'_4 u \Big( \varepsilon + \sup_{Q \in \cR(x^n)} Q\Big( \sup_{s \in [0,u]} \|X_s - x^n\| > \delta/2\Big)\Big),
    \end{split}
\end{equation}
for some constant \(C'_4 > 0\) that is independent of \(n \in \bN\).

In summary, we conclude from \eqref{eq: bound i^n} and \eqref{eq: bound j_1} -- \eqref{eq: bound j_4} that there exists a constant \(C >0\), independent of \(n \in \bN\), such that 
\begin{equation} \label{eq: bound i^n 2}
    \begin{split}
        	I^n_u \leq - u \partial_t & \phi (t^n, x^n) + E^P \Big[ \int_0^{u \wedge \tau^n_r} G^n (X_s) ds \Big]  \\&+ C u\Big( \varepsilon + \sup_{Q \in \cR(x^n)} Q\Big( \sup_{s \in [0,u]} \|X_s - x^n\| > \delta/2\Big)\Big).
    \end{split}
\end{equation}
We have, for all \(x \in \bR^d\),  
\begin{align*}
	G^n (x) & \leq G (x^n, \phi(t^n,\cdot\,)) + C \sup \Big\{ \| b (f, x)- b(f,x^n) \| + \tfrac{1}{2} \|\sigma \sigma^* (f, x) - \sigma \sigma^*(f,x^n) \| 
 \\& \qquad + \int \big| \g(t^n,x^n, k(f,x,z) - \g(t^n,x^n, k(f, x^n,z)) \big| \, \n(f, dz) \colon f \in F \Big\}, 
\end{align*}
for some constant \(C>0\) that is independent of \(n\in\bN\).
Define
\begin{equation} \label{eq: def j_5}
\begin{split}
   J^n_5(u,P) & := E^P \Big[ \int_0^{u \wedge \tau^n_r}  \sup_{f \in F} \Big( \| b (f, X_s)- b(f,x^n) \| + \tfrac{1}{2} \|\sigma \sigma^* (f, X_s) - \sigma \sigma^*(f,x^n) \| 
 \\& \qquad + \int \big| \g(t^n,x^n, k(f,X_s,z) - \g(t^n,x^n, k(f, x^n,z)) \big|\, \n(f, dz) \Big) ds 
    \Big]
\end{split}
\end{equation}
Again, we estimate, for \(f \in F\) and \(s < u \wedge \tau^n_r\),
\begin{align*}
    \| b (f, X_s)- b(f,x^n) \| & \leq \varepsilon + 2 \sup_{f \in F} \sup_{\|y-x^n\| \leq r} \|b(f,y)\| \1_{\{\|X_s - x^n \| \, >\, \delta/2\}},
    \\  \|\sigma \sigma^* (f, X_s) - \sigma \sigma^*(f,x^n) \| & \leq \varepsilon  + 2 \sup_{f \in F} \sup_{\|y-x^n\| \leq r} \|\sigma \sigma^*(f,y)\| \1_{\{\|X_s - x^n \| \, >\, \delta/2\}}.
\end{align*}
Next, we deal with the non-local part. For every \((f,z) \in F \times \bR^d\)  and \(s < u \wedge \tau^n_r\), using Lemma \ref{lem: non local 1}, we get that
\begin{align*}
\int  \big| \g(t^n&,x^n, k(f,X_s,z) - \g(t^n,x^n, k(f, x^n,z)) \big| \, \n(f, dz)
\\& \leq \Big( \varepsilon + 2 C \Big(\sup_{f\in F} \sup_{\|y - x^n\| \leq r} \int (\| k(f,y,z) \|^2 \wedge 1) \, \n(f, dz) \Big) \1_{\{\|X_s - x^n \| \,>\, \delta/2\}} \Big).
\end{align*}
Hence, 
\begin{equation} \label{eq: bound j_5}
\begin{split}
 J^n_5(u,P) \leq C u \Big( \varepsilon + \sup_{Q \in \cR(x^n)} Q\Big( \sup_{s \in [0,u]} \|X_s - x^n\| > \delta/2\Big)\Big).    
 \end{split}
\end{equation}
We conclude that 
\begin{align*}
	E^P \Big[\int_0^{u \wedge \tau^n_r} G^n (X_s) ds \Big] \leq u G (t^n, x^n, \phi) + C u \Big( \varepsilon + \sup_{Q \in \cR(x^n)} Q\Big( \sup_{s \in [0,u]} \|X_s - x^n\| > \delta/2\Big)\Big),
\end{align*}
and therefore, combining \eqref{eq: bound i^n 2} with this observation, we have
\begin{equation} \label{eq: final estimate i^n}
\begin{split}
I^n_u \leq - u  \partial_t \phi (&t^n, x^n) + u G (t^n, x^n, \phi) \\&+ C u\Big( \varepsilon + \sup_{n \in \mathbb{N}} \sup_{Q \in \cR(x^n)} Q\Big( \sup_{s \in [0,u]} \|X_s - x^n\| > \delta/2\Big)\Big). 
\end{split}
\end{equation}
Next, observe that 
\begin{equation*} 
\big| E^P \big[ \phi (t^n - (u \wedge \tau^n_r), X_{u \wedge \tau^n_r}) \big]  -   E^P \big[ \phi (t^n - u , X_{u}) \big] \big| \leq 2 \| \phi \|_\infty P \Big( \sup_{s \in [0,u]} \|X_s - x^n \| > r \Big).
\end{equation*}
Since \(P \in \cR(x^n) \) was arbitrary, we conclude, with Corollary~\ref{cor: maximal inequality}, that 
\begin{equation} \label{eq: estimate tau}
\begin{split}
\Big| \sup_{Q \in \cR(x^n)} & E^Q\big[ \phi (t^n - (u \wedge \tau^n_r), X_{u \wedge \tau^n_r}) \big] - \sup_{Q \in \cR(x^n)} E^Q\big[ \phi (t^n - u, X_u) \big]\Big| 
\\&\leq 2 \| \phi \|_\infty \sup_{Q \in \cR(x^n)} Q \Big( \sup_{s \in [0,u]} \|X_s - x^n \| > r \Big)
\\& \leq u C_r
\end{split}
\end{equation}
for some constant \(C_r > 0\) such that \(C_r \to 0\) as \(r \to \infty\).

Recalling \eqref{eq: main ev upper} and combining \eqref{eq: final estimate i^n} with \eqref{eq: estimate tau}, we obtain
\begin{equation} \begin{split}
	0 \leq  - u \partial_t & \phi (t^n, x^n) +  u G (t^n, x^n, \phi) + uC_r \phantom \int \\&+ C u\Big( \varepsilon + \sup_{n \in \mathbb{N}} \sup_{Q \in \cR(x^n)} Q\Big( \sup_{s \in [0,u]} \|X_s - x^n\| > \delta/2\Big)\Big) 
       \\& + \phi (t^n, x^n) - \phi (t^0, x^0) + v^* (t^0, x^0) - v (t^n, x^n).
\end{split}
\end{equation} 
Thanks to Lemma \ref{lem: G cont}, the map \((t, x) \mapsto G (x, \phi(t,\cdot\,))\) is continuous. 
Hence, taking the limit \(n \to \infty\), we obtain that 
\begin{align*}
	0 \leq   -u \partial_t & \phi (t^0, x^0) + u G (x^0, \phi(t^0,\cdot\,)) + uC_r \\&+ C u\Big( \varepsilon + \sup_{n \in \mathbb{N}} \sup_{Q \in \cR(x^n)} Q\Big( \sup_{s \in [0,u]} \|X_s - x^n\| > \delta/2\Big)\Big), 
\end{align*}
Dividing the right hand side by \(u\) and then letting \(u \searrow 0\),
we obtain, with Corollary \ref{cor: maximal inequality}, that 
\begin{align*}
	0 \leq -\partial_t \phi (t^0, x^0) + G (x^0, \phi(t^0,\cdot\,)) + C_r + C \varepsilon.
\end{align*}
Finally, taking first \(\varepsilon \to 0\) and then \(r \to \infty\), we conclude
\begin{align*}
	0 \leq -\partial_t \phi (t^0, x^0) + G (x^0, \phi(t^0,\cdot\,)).
\end{align*}
This proves that \(v^*\) is a viscosity subsolution to \eqref{eq: G eq}. 
\end{proof}

\subsection{Supersolution Property} \label{sec: pf supersolution}
In this section, we establish the supersolution property of~\(v_*\), i.e., part (ii) of Theorem~\ref{thm: main V}.
\begin{lemma} \label{lem: en sup}
Suppose that Condition~\ref{cond: V1} holds. Then, \(v_*\) is a viscosity supersolution to the nonlinear PDE \eqref{eq: G eq}.
\end{lemma} 
\begin{proof}
First, observe that \(v_*(0,x) = \psi(x)\) by Lemma \ref{lem: initial values}.
Let \(\phi \in  C^\infty_b(\bR_+ \times \bR^d; \bR)\) such that \(\phi \leq v_*\) and \(\phi (t^0, x^0) = v_*(t^0, x^0)\) for some \((t^0, x^0) \in (0, \infty) \times \bR^d\). 
There exists a sequence \((t^n, x^n)_{n = 1}^\infty \subset \bR_+ \times\bR^d\) such that \((t^n, x^n) \to (t^0, x^0)\) and 
\[
v_* (t^0, x^0) = \lim_{n \to \infty} v (t^n, x^n). 
\]
We take an arbitrary \(0 < u < t_* := \inf_{n \in \mathbb{Z}_+} t^n\), where, w.l.o.g., \(t_* > 0\).
Similar to \eqref{eq: main ev upper}, we obtain that 
\begin{equation} \label{eq: main ev lower}
	\begin{split}
		0\geq \sup_{P \in \cR (x^n)} &E^P \big[ \phi (t^n - u, X_u) \big] - \phi (t^n, x^n) \\&+ \phi (t^n, x^n) - \phi (t^0, x^0) + v_* (t^0, x^0) - v (t^n, x^n).
	\end{split}
\end{equation}
Fix \(f \in F\) and \(n \in \bN\). 
Thanks to Proposition \ref{prop: existence}, there exists a measure \(P = P_f \in \cR(x^n)\) such that the Lebesgue densities of the \(P\)-characteristics of \(X\) coincide \((\llambda \otimes P)\)-a.e. with 
\begin{equation} \label{eq: P_f}
  (b^P, a^P, \nu^P) = \big(b (f, X), \sigma \sigma^* (f, X), \n(f, \, \cdot \,) \o k (f, X, \, \cdot \,)^{-1} \big).
\end{equation}
For \(r > 0\) and \(n \in \mathbb{Z}_+\), we define the stopping times 
\begin{align*}
\tau_r &:= \inf \Big\{t \geq 0 \colon \sup_{u \in [0, t]} \|X_u \| \geq r \Big\}, \\
\tau^n_r &:= \inf \Big\{t \geq 0 \colon \sup_{u \in [0, t]} \|X_u - x^n \| \geq r \Big\} = \tau_r (X - x^n).
\end{align*}
Notice that 
\[
\{\tau_r \leq t \} = \Big\{ \sup_{u \in [0, t]} \| X_u \| \geq r \Big\}.
\]
The map \(\omega \mapsto \sup_{u \in [0, s]} \| \omega (u) \|\) is upper semicontinuous in the Skorokhod \(J_1\) topology (see \cite[Theorem~15.20]{HWY}). Hence, the set \(\{\tau_r \leq t\}\) is closed and therefore, \(\omega \mapsto \tau_r (\omega)\) is lower semicontinuous, which yields that 
\begin{align} \label{eq: ST upper semi}
\liminf_{n \to \infty} \tau^n_r \geq \tau^0_r.
\end{align}
It\^o's formula implies that \(P\)-a.s.
\begin{align*}
\phi (t^n - (u \wedge \tau^n_r), X_{u \wedge \tau^n_r}) & - \phi (t^n, x^n)  = - \int_0^{u \wedge \tau^n_r} \partial_t \phi (t^n - s, X_s) ds
\\&+ \int_0^{u \wedge \tau^n_r} \Big[ \langle b^P_s,  \nabla \phi (t^n - s, X_s) \rangle + \tfrac{1}{2} \on{tr} \big[ \nabla^2 \phi (t^n - s, X_s) a^P_s \big]   \Big] ds
\\&+ \int_0^{u \wedge \tau^n_r} \int  \g(t^n-s, X_{s-}, z) \mu^X(ds,dz)
\\&+ \text{local \(P\)-martingale}, \phantom \int
\end{align*}
where the function \(\g\) is defined in \eqref{eq: def g}.
Thanks to the local boundedness Condition \ref{cond: V1} and the definition of \(\tau^n_r\), 
it follows that the local \(P\)-martingale part is actually a true \(P\)-martingale.
Hence, for all \(n \in \mathbb{N}\), we obtain that 
\begin{equation*} \begin{split}
E^P & \big[  \phi (t^n - u \wedge \tau^n_r, X_{u \wedge \tau^n_r}) - \phi (t^n, x^n) \big] 
\\&= E^P \Big[ \int_0^{u \wedge \tau^n_r} \big( - \partial_t \phi (t^n - s, X_s) + \langle b^P_s, \nabla \phi (t^n - s, X_s) \rangle + \tfrac{1}{2} \on{tr} [ \nabla^2 \phi (t^n - s, X_s) a^P_s] \big) ds \Big]
\\&\qquad \qquad + E^P \Big[ \int_0^{u \wedge \tau^n_r} \int \g(t^n - s, X_{s}, z) \nu^P(ds,dz) \Big]
\\&=: I^n_u. \phantom \int
\end{split}
\end{equation*}
We introduce, for \((t,x) \in \bR_+ \times \bR^d\), the operators
\begin{align*}
G_f(x, \phi(t,\cdot\,)) := \langle  \nabla \phi (t,x), b (f, x) \rangle
    &+ \tfrac{1}{2} \on{tr} \big[ \nabla^2\phi (t, x) \sigma \sigma^* (f, x) \big]
     \\&+ \int\, \g(t,x,k(f,x,z)) \, \n (f, dz)
\end{align*}
and
\begin{align*}
  G^n_f (x) :=  \langle \nabla  \phi (t^n, x^n), b (f, x) \rangle & + \tfrac{1}{2} \on{tr} \big[ \nabla^2 \phi (t^n, x^n) \sigma \sigma^*(f,x) \big] 
   \\&+ \int \g (t^n, x^n, k (f, x, z)) \, \n (f, dz) 
\end{align*}
Notice that 
\begin{equation} \label{eq: estimate G_f G^n_f}
\begin{split}
   G_f(x^n,\phi(t^n,\cdot\,)) \leq G^n_f (x) &+  C \Big( \| b (f, x)- b(f,x^n) \| + \tfrac{1}{2} \|\sigma \sigma^* (f, x) - \sigma \sigma^*(f,x^n) \| 
 \\& + \int \big| \g(t^n,x^n, k(f,x,z)) - \g(t^n,x^n, k(f, x^n,z)) \big| \, \n(f, dz) \Big),
\end{split}
\end{equation}
for some constant \(C>0\) that is independent of \(n\in\bN\).
Furthermore, recalling the definition of the quantities \(J^n_k(u,P)\) for \(k=2,3,4\) from \eqref{eq: def j^n}, we deduce from \eqref{eq: P_f} that
\begin{equation} \label{eq: estimate G^n_f}
\begin{split}
   E^P \Big[\int_0^{u \wedge \tau^n_r} G^n_f(X_s) ds \Big] &\leq  E^P \Big[ \int_0^{u \wedge \tau^n_r} \langle b^P_s, \nabla \phi (t^n + s, X_s) \rangle + \tfrac{1}{2} \on{tr} \big [ \nabla^2 \phi (t^n + s, X_s) a^P_s \big]  ds
   \Big]
    \\&\qquad\qquad + E^P \Big[ \int_0^{u \wedge \tau^n_r} \int \g(t^n+s, X_{s}, z) \nu^P(ds,dz) \Big]
     \\&\qquad\qquad + J^n_2(u,P) + \tfrac{1}{2} J^n_3(u,P) + J^n_4(u,P). \phantom \int
\end{split}
\end{equation}
Moreover,
\begin{equation} \label{eq: estimate partial_t}
\begin{split}
- E^P \big[u \wedge \tau^n_r\big] \partial_t \phi(t^n,x^n) &= -E^P\Big[\int_0^{u \wedge \tau^n_r} \partial_t\phi(t^n,x^n) ds \Big] 
\\&\leq - E^P\Big[ \int_0^{u \wedge \tau^n_r} \partial_t \phi(t^n - s,X_s) ds \Big] + J^n_1(u,P),
\end{split}
\end{equation}
where \(J^n_1(u,P)\) is defined in \eqref{eq: def j^n}.
Combining \eqref{eq: estimate G_f G^n_f} -- \eqref{eq: estimate partial_t}, we conclude with the help of \eqref{eq: bound j_1} -- \eqref{eq: bound j_4} and \eqref{eq: bound j_5}, where \(\varepsilon\) and \(\delta\) are as in the proof of Lemma~\ref{lem: en sub}, that
\begin{equation*}
    \begin{split}
       I^n_u & \geq  E^P \big[u \wedge \tau^n_r\big]  \big( - \partial_t \phi(t^n,x^n) + G_f(x^n,\phi(t^n,\cdot\, ))  \big)
       \\& \hspace{3cm}- \big(J^n_1(u,P) + J^n_2(u,P) + J^n_3(u,P) + J^n_4(u,P) + J^n_5(u,P) \big)
       \\& \geq E^P \big[u \wedge \tau^n_r \big] \big( - \partial_t \phi(t^n,x^n) + G_f(x^n,\phi(t^n,\cdot\,)) \big) \\&\hspace{3cm}- C u\Big( \varepsilon + \sup_{n \in \mathbb{N}} \sup_{Q \in \cR(x^n)} Q\Big( \sup_{s \in [0,u]} \|X_s - x^n\| > \delta/2\Big)\Big).
    \end{split}
\end{equation*}
Recalling \eqref{eq: main ev lower} and \eqref{eq: estimate tau}, we obtain that 
\begin{align*}
	0 &\geq   E^P \big [u \wedge \tau^n_r \big] \big( - \partial_t \phi (t^n, x^n) + G_f(x^n, \phi(t^n,\cdot\,)) \big) \\&\hspace{3cm}- C_r u - C u\Big( \varepsilon + \sup_{n \in \mathbb{N}} \sup_{Q \in \cR(x^n)} Q\Big( \sup_{s \in [0,u]} \|X_s - x^n\| > \delta/2\Big)\Big) 
       \\&\hspace{3cm} + \phi (t^n, x^n) - \phi (t^0, x^0) + v^* (t^0, x^0) - v (t^n, x^n), 
\end{align*}
for some constant \(C_r > 0\) such that \(C_r \to 0\) as \(r \to \infty\).
Observe that the map \((t, x) \mapsto G_f (x, \phi(t,\cdot\,))\) is continuous (cf. Lemma \ref{lem: G cont}). 
Recalling \eqref{eq: ST upper semi}, we obtain, by taking the limes inferior and Fatou's lemma, that
\begin{equation} \label{eq: pf super}
    \begin{split}
        0  & \geq \liminf_{n \to \infty} E^P \big [u \wedge \tau^n_r \big]  \big( - \partial_t \phi (t^0, x^0) +  G_f(x^0, \phi(t^0,\cdot\,)) \big) 
\\&\hspace{3cm} -C_r u - C u\Big( \varepsilon + \sup_{n \in \mathbb{N}} \sup_{Q \in \cR(x^n)} Q\Big( \sup_{s \in [0,u]} \|X_s - x^n\| > \delta/2\Big)\Big)
\\& \geq E^P \big [u \wedge \tau^0_r \big]  \big( - \partial_t \phi (t^0, x^0) +  G_f(x^0, \phi(t^0,\cdot\,)) \big) 
\\&\hspace{3cm} -C_r u - C u\Big( \varepsilon + \sup_{n \in \mathbb{N}} \sup_{Q \in \cR(x^n)} Q\Big( \sup_{s \in [0,u]} \|X_s - x^n\| > \delta/2\Big)\Big).
    \end{split}
\end{equation}
Since \(P( \tau^0_r = 0 ) = P (\|X_0 - x^0\| \geq r) = 0\), using the dominated convergence theorem, we get that 
\begin{align*}
 \lim_{u \searrow 0} \frac{E^P \big[ u \wedge \tau^0_r \big]}{u} = 1.
\end{align*}
Thus, dividing the last term in \eqref{eq: pf super} by \(u\) and then letting \(u \searrow 0\),
we obtain with Corollary~\ref{cor: maximal inequality} that 
\[
	0 \geq - \partial_t \phi (t^0, x^0) + G_f(x^0, \phi(t^0,\cdot\,))  - C_r  - C \varepsilon.
\]
Taking the limit \(\varepsilon \to 0\) and then \(r \to \infty\), we conclude
\begin{align*}
	0 \geq - \partial_t \phi (t^0, x^0) + G_f(x^0, \phi(t^0,\cdot\,)).
\end{align*}
Finally, taking the sup over all \(f \in F\) proves that \(v_*\) is a viscosity supersolution to \eqref{eq: G eq}. 
\end{proof}

\subsection{Proof of Theorem \ref{thm: main V}} \label{sec: pf thm viscosity}
Part (i) of Theorem~\ref{thm: main V} follows from Lemma \ref{lem: en sub} and part (ii) follows from Lemma \ref{lem: en sup}. Taking (i) and (ii) into consideration, part (iii) is a consequence of the comparison principle that is given by Theorem \ref{theo: comparison}.\qed

\appendix

\section{A Comparison Result for L\'evy-type HJB Equations} \label{app: comparison}
The purpose of this appendix is to tailor the abstract comparison result given by \cite[Corollary~2.23]{hol16} to our framework. 

\begin{condition} \label{cond: local bdd}
	The functions \(\bb\) and \(\bs\) satisfy the following boundedness condition:
	\[
	\forall x \in \bR^d \colon \quad \sup \Big\{ \|\bb (f, x)\| + \|\bs(f, x)\| \colon f \in F \Big\} < \infty.
	\]
 	Moreover, there is a constant \(C > 0\) and a Borel function \(\gamma \colon L \to [0, \infty]\) such that 
  \[\sup_{f \in F} \int (\gamma^2(z) \wedge 1)\, \n (f, dz) < \infty\] 
  and 
\begin{align*}
    \lim_{\kappa \searrow 0} \, \sup_{f \in F} \int_{\{\gamma\leq \kappa \}} \gamma^2 (z)\, \n (f, dz) &= 0, \quad
    \lim_{R \to \infty} \, \sup_{f \in F} \n \, (f, \{\gamma > R\}) = 0.
\end{align*}
  as well as
	\begin{align*}
		 \| \bs(f,x)-\bs(f,y) \| &\leq C  \|x -y \|,  \\
		\| \bk (f, x, z) - \bk (f, y, z)\| &\leq  \gamma(z)  \|x - y\|, \\
		\| \bk (f, x, z)\| & \leq  \gamma(z) (1 + \|x\|)
	\end{align*}
	for all \(f \in F\), \(x,y \in \bR^d\) and \(z \in L\).
 Define \(\ob \colon F \times \bR^d \to \bR^d\) by    
 \begin{align*}
    \ob (f, x) := b (f, x) &+ \int_{\{\gamma \leq 1\}} \big( k (f, x, z) - h (k (f, x, z)) \big)\, \n (f, dz) \\&- \int_{\{\gamma > 1\}} h (k (f, x, z)) \, \n (f, dz).
    \end{align*}
    As explained in Appendix~\ref{app: well-defined} below, the function \(\ob\) is well-defined under the above conditions. Finally, there exists a constant \(C > 0\) such that
    \begin{align*}
        \| \ob (f, x) - \ob (f, y) \| \leq C \|x - y\|
    \end{align*}
    for all \(f \in F\) and \(x, y \in \bR^d\).
\end{condition}

\begin{theorem} \label{theo: comparison}
Suppose that Condition~\ref{cond: local bdd} holds, that \(u \colon \bR_+ \times \bR^d \to \bR\) is a bounded viscosity subsolution and that \(v \colon \bR_+ \times \bR^d \to \bR\) is a bounded viscosity supersolution to the nonlinear PDE \eqref{eq: G eq}. Then, \(u \leq v\).
\end{theorem}

\begin{proof}
	Using arguments as in \cite[Lemma~2.32, Proposition~2.33]{hol16}, the result can be deduced from \cite[Corollary~2.23]{hol16}. For brevity, we provide not all details. Instead, we sketch the idea that relates our setting to the HJB framework from \cite[Definition~2.27]{hol16}. Let \(\gamma \colon L \to [0,\infty]\) be as in Condition~\ref{cond: local bdd}. For \(\phi \in C^2_b (\bR^d; \bR)\), set 
	\begin{align*}
		J^f (x) := \int\, \big[ \phi(x + \bk (f, x, z)) - \phi(x) - \langle \nabla \phi(x) , h (\bk (f, x, z)) \rangle\big]\, \n (f, dz).
	\end{align*}
	Taking some \(\kappa \in (0, 1)\), we split the operator \(J^f\) as follows: 
	\begin{align*}
		J^f (x) &= H^f (x) + \langle \nabla \phi (x), K^f (x) \rangle, \phantom \int 
	\end{align*}
with 
\begin{align*}
H^f (x) &:= \int_{\{\gamma \leq \kappa\}} \big[ \phi(x + \bk (f, x, z)) - \phi(x) - \langle \nabla \phi(x) , \bk (f, x, z) \rangle\big]\, \n (f, dz) 
		\\&\qquad \qquad + \int_{\{ \kappa < \gamma \leq 1\}} \big[ \phi(x + \bk (f, x, z)) - \phi(x) - \langle \nabla \phi(x) , \bk (f, x, z) \rangle\big]\, \n (f, dz)
		\\& \qquad \qquad + \int_{\{ \gamma > 1\}} \big[ \phi(x + \bk (f, x, z)) - \phi(x) \big]\, \n (f, dz),
  \\
K^f (x) &:=  \int_{\{\gamma \leq 1\}} \big(\bk (f, x, z) - h (\bk (f, x, z)) \big) \, \n (f, dz) - \int_{\{\gamma > 1\}}  h (\bk (f, x, z)) \, \n (f, dz).
\end{align*}
It is shown in Appendix~\ref{app: well-defined} that \(K^f\) is well-defined under Condition~\ref{cond: local bdd}.
Returning to the decomposition \(J^f = H^f + \langle \nabla \phi, K^f\rangle\), we notice that \(H^f\) is of a similar form as the HJB operator from \cite[Definition~2.27]{hol16} with the exception that \(\n (f, dz)\) is not defined on \(\mathcal{B}(\bR^d)\) but on \(\mathcal{B}(L)\), and the decomposition is not done with the sets \(\{|z| \leq \kappa\}, \{\kappa < |z| \leq 1\}\) and \(\{|z| > 1\}\) but with \(\{\gamma \leq \kappa\}, \{\kappa < \gamma \leq 1\}\) and \(\{\gamma > 1\}\). A careful inspection of the proofs for \cite[Lemma~2.32, Proposition~2.33]{hol16} shows that the argument for the validity of the prerequisites of \cite[Corollary~2.23]{hol16} only needs cosmetic modifications (namely, \(|z|\) in the estimates from \cite{hol16} has to be replaced by \(\gamma (z)\) in our setting). 
The term \(K^f\) can be added to the drift coefficient \(\bb\), which leads to the modified coefficient \(\ob\). As Condition~\ref{cond: local bdd} contains a Lipschitz assumption for \(\ob\), the modified drift coefficient can be treated as in \cite{hol16}.
Omitting full details of the argument sketched above for brevity, we conclude the claim of Theorem~\ref{theo: comparison} from \cite[Corollary~2.23]{hol16}.
\end{proof}

\section{On the modified drift coefficient} \label{app: well-defined}
Take a compact set \(K \subset \bR^d\) and suppose that there exists a Borel function \(\gamma = \gamma\, (K) \colon L \to \bR_+\), such that 
\begin{align*}
    \sup_{f \in F} \int (\gamma^2 (z) \wedge 1) \, \n (f, dz) < \infty, 
\end{align*}
and 
\begin{align*}
    \lim_{\kappa \searrow 0} \sup_{f \in F} \int_{\{\gamma \leq \kappa\}} \gamma^2(z) \, \n(f,dz) = 0, \quad \lim_{R \to \infty} \sup_{f \in F} \n \, (f, \{\gamma > R\}) = 0.
\end{align*}
Further, assume that there exists a modulus of continuity \(\ka = \ka\, (K) \colon [0, \infty] \to [0, \infty]\) such that 
\begin{align} \label{eq: conti assp}
          \|k (f, x, z) - k (f, y, z)\| &\leq \gamma (z)\, \ka\, (\|x - y\|),
          \quad \|k (f, x, z) \| \leq \gamma (z)
\end{align}
for all \(f \in F\), \(x, y \in K\) and \(z \in L\).
Under these assumptions, we have the following:
      \begin{lemma} \label{lem: lemma cont modi drift}
     For all \(f \in F\) and \(x \in K\), 
     \begin{align*}
     \int_{\{\gamma \leq 1\}} \| k (f&, x, z) - h (k (f, x, z)) \| \, \n (f, dz) \\&+ \int_{\{\gamma > 1\}} \| h (f, x, z) \| \, \n (f, dz) \leq C \, \int (\gamma^2 (z) \wedge 1 ) \, \n (f, dz).
     \end{align*}
     Furthermore, there exists a modulus of continuity \(\ka^* = \ka^* (K) \colon [0, \infty] \to [0, \infty]\) such that 
     \begin{align*}
         \int_{\{\gamma \leq 1\}} \big\| k (f, x, z) - h (k (f, x, z)) - ( k (f, y, z) - h (k (f, y, z)))\big\| \, \n (f, dz) &\leq \ka^* (\|x - y\|), \\ 
         \int_{\{\gamma > 1\}} \big\| h (k (f, x, z)) - h (k (f, y, z))\big\| \, \n (f, dz) &\leq \ka^* (\|x - y\|)
     \end{align*}
     for all \(f \in F\) and \(x, y \in K\).
      \end{lemma}
      \begin{proof}
Regarding the first claim, observe that, for \(z \in \{\gamma \leq 1\}\) and all \(f \in F, \, x \in K\),
\begin{equation} \label{eq: first bound modi chara}
\begin{split}
\| \bk (f, x, z) - h (\bk (f, x, z))\| &\leq C \|\bk (f, x, z)\| \1_{\{\|\bk (f, x, z)\|\, >\, \epsilon\}} \\&\leq C \|\bk (f, x, z)\|^2 
\\&\leq C ( \gamma^2 (z) \wedge 1),
\end{split}
\end{equation}
where \(\epsilon > 0\) is such that \(h (y) = y\) for \(\|y\| \leq \epsilon\). Similarly, for the second term, we record that
\[
\|h (\bk(f,x,z))\| \1_{\{\gamma(z)\, >\, 1\}} \leq C\, ( \gamma^2(z) \wedge 1 ), \quad (f,x,z) \in F \times \bR^d \times L,
\]
where the constant \(C > 0\) bounds the truncation function \(h\). 

\smallskip
Next, we discuss the equi-continuity assertions. First, fix \(\varepsilon \in (0, 1)\) and \(x, y \in K\). Take \(\kappa \in (0, 1)\) such that 
\[
\sup_{f \in F} \int_{\{\gamma \leq \kappa\}} \gamma^2(z) \, \n (f, dz)  \leq \varepsilon.  
\]
Then, using \eqref{eq: first bound modi chara}, we get that 
\begin{align*}
    \int_{\{\gamma \leq \kappa\}} \big\| k (f, x, z) - h (k (f, x, z)) - ( k (f, y, z) - h (k (f, y, z)))\big\| \, \n (f, dz) \leq C \, \varepsilon. 
\end{align*}
For the remainder, using that \(h\) is Lipschitz continuous and \eqref{eq: conti assp}, we obtain that 
\begin{align*}
     \int_{\{\kappa < \gamma \leq 1\}} \big\| k (f, x, z) - h (k (f, x, z)) - ( k (f, y, z) & - h (k (f, y, z)))\big\| \, \n (f, dz) 
     \\&\leq \ka (\|x - y\|)\, \frac{C}{\kappa}\, \int ( \gamma^2 (z) \wedge 1) \, \n (f, dz). 
\end{align*}
This proves equi-continuity of 
\[
\Big\{ K \ni x \mapsto  \int_{\{\gamma \leq 1\}} \big( k (f, x, z) - h (k (f, x, z)) \big) \, \n (f, dz) \colon f \in F \Big\}, 
\]
which entails the existence of a modulus of continuity as claimed. 

We proceed with the final assertion. First, take \(R > 1\) such that 
\begin{align*}
    \sup_{f \in F} \n \, (f, \{\gamma > R\}) \leq \varepsilon. 
\end{align*}
Now, using again the Lipschitz continuity of \(h\) and \eqref{eq: conti assp}, we obtain that 
\begin{align*}
     \int_{\{\gamma > 1\}} \big\| h (k (f&, x, z)) - h (k (f, y, z))\big\| \, \n (f, dz) 
     \\&\leq 2 \|h\|_\infty \, \varepsilon +  \int_{\{1 < \gamma \leq R\}} \big\| h (k (f, x, z)) - h (k (f, y, z))\big\| \, \n (f, dz)
     \\&\leq 2 \|h\|_\infty\, \varepsilon + \ka (\|x - y\|) \, R \ \n (f, \{\gamma > 1\})
     \phantom \int
      \\&\leq 2 \|h\|_\infty\, \varepsilon + \ka (\|x - y\|) \, R \, \int (\gamma^2 (z) \wedge 1) \, \n (f, dz). \phantom \int
\end{align*}
This establishes equi-continuity of 
\[
\Big\{ K \ni x \mapsto  \int_{\{\gamma \leq 1\}} h (k (f, x, z)) \, \n (f, dz) \colon f \in F \Big\}, 
\]
and therefore, completes the proof. 
      \end{proof}

\end{document}